\documentclass[a4paper,final,texthight24cm,11pt]{article}

\usepackage{amssymb}
\usepackage{graphicx}
\usepackage{amscd}
\usepackage{amsmath,amsthm,amsxtra}
\theoremstyle{plain}
\newtheorem{prop}{Proposition}[section]

\newtheorem{thm}[prop]{Theorem}

\theoremstyle{definition}
\newtheorem{defi}[prop]{Definition}

\theoremstyle{plain}

\newtheorem{rem}[prop]{Remark}

\numberwithin{equation}{section}

\title{Quasi regular Dirichlet forms and the stochastic quantization problem}
\author{Sergio Albeverio \thanks{Inst.\ Appl.\ Mathematics and HCM, University of Bonn; CERFIM (Locarno); albeverio@uni-bonn.de} \thanks{BiBoS (Bielefeld, Bonn)} \and Zhi Ming Ma \thanks{Inst.\ Appl.\ Math., AMSS, CAS, Beijing} \and Michael R{\"o}ckner \footnotemark[2] \thanks{Mathematics Faculty, University of Bielefeld}}
\date{Dedicated to Masatoshi Fukushima for his $80^{\mathrm{th}}$ birthday}

\begin{document}

\maketitle
\centerline{\date{7.4.14}}
\mbox{}
\begin{abstract}
	After recalling basic features of the theory of symmetric quasi regular Dirichlet forms we show how by applying it to the stochastic quantization equation, 
	with Gaussian space-time noise, one obtains weak solutions in a large invariant set. Subsequently, we discuss non symmetric quasi regular Dirichlet forms and 
	show in particular by two simple examples in infinite dimensions that infinitesimal invariance, does not imply global invariance. We also present a simple example of non-Markov uniqueness in infinite dimensions.
\end{abstract}

%%%%%%%%%%%%%%%%%%%%%%%%%%%%%%%%%%%%%%%%%%%%%%%%%%%%%%%%%%%%%%%%%%%%%%%%%%%%%%%%%%%%%%%%%%%%%%%%%%%%
%%%%%%%%%%%%%%%%%%%%%%%%%%%%%%%%%%%%%%%%%%%%%%%%%%%%%%%%%%%%%%%%%%%%%%%%%%%%%%%%%%%%%%%%%%%%%%%%%%%%
%%%%%%%%%%%%%%%%%%%%%%%%%%%%%%%%%%%%%%%%%%%%%%%%%%%%%%%%%%%%%%%%%%%%%%%%%%%%%%%%%%%%%%%%%%%%%%%%%%%%

\section{Introduction}

The theory of symmetric Dirichlet forms is the natural extension to higher dimensional state spaces of the classical Feller theory of stochastic processes on the real line. Through ground breaking work by Beurling and Deny (1958-59), \cite{BD58}, \cite{BD59}, \cite{Den70}, \cite{Sil74}, \cite{Sil76}, and Fukushima (since 1971), \cite{Fu71a}, \cite{Fu71b}, \cite{Fu80}, \cite{FOT11}, \cite{CF12}, it has developed into a powerful method of combining analytic and functional analytic, as well as potential theoretic and probabilistic methods to handle the basic construction of stochastic processes under natural conditions about the local characteristics (e.g., drift and diffusion coefficients), avoiding unnecessary regularity conditions. 

Such processes arise in a number of applied problems where coefficients can be singular, locally or at infinity. For detailed expositions of the general theory, mainly concentrated on finite dimensional state spaces, see \cite{Fu80}, \cite{F82}, \cite{FOT11}, \cite{Ma95} and, e.g., \cite{Alb03}, \cite{BBCK09}, \cite{E99}, \cite{Fu10}, \cite{GriH08}, \cite{Hi10}, \cite{KaYa05}, \cite{KKVW09}, \cite{KT91}. For further new developments, e.g.\ concerning boundary conditions, or non symmetric processes, or processes with jumps, see, e.g., \cite{CF12}, \cite{CMR94}, \cite{AFH11}, \cite{AMU98}, \cite{AU00}, \cite{Sta99}, \cite{Sta99b}, \cite{BiT07}, \cite{Bou03}, \cite{KS03}, \cite{LW99}, \cite{Osh04}, \cite{Jac01}, \cite{JS00}, \cite{SU07}. 
An extension of the theory to processes with infinite dimensional state spaces of Lusin type has been first described indirectly, by a suitable map into a larger locally compact space and consequent reductions to the finite dimensional case, by Fukushima \cite{Fu71a}, \cite{Fu71b} and then further developed in \cite{Fu80}, \cite{Fu92}, \cite{FH01}, \cite{FOT11}, \cite{ST92}. This has been called ``regularization method'' (see \cite{Fu80}, \cite{AMR93}, \cite{ST92}, \cite{Alb03}). Another approach on Wiener space has been developed by \cite{BH91} and on certain Banach spaces by \cite{FeyLa91}. Mainly motivated by applications to infinite dimensional processes connected with SPDE's, like those arising in quantum field theory, see, e.g., \cite{Gro76}, \cite{AHK74}, \cite{AHK76}, \cite{AHK77a}, \cite{AHK77b}, \cite{AHKS77}, \cite{Alb97}, \cite{N73}, \cite{PaWu81}, an extension of Dirichlet form theory to infinite dimensional state spaces of more general type, including spaces of distributions, has been developed in \cite{AHK76}, \cite{AHK77a}, \cite{AHK77b}, \cite{AM91}, \cite{AM92}, \cite{AMR92a}, \cite{AMR92b}, \cite{AMR92c}, \cite{AMR93}, \cite{AFHMR92}, \cite{AR89a}, \cite{AR90b}, \cite{Ku}, \cite{Kus82}. 

This theory is now known as the theory of quasi-regular Dirichlet forms and a systematic exposition of it is in \cite{MR92}. For newer developments see also \cite{AFHMR92}, \cite{Aid00}, \cite{Alb03}, \cite{AR05}, \cite{ARW}, \cite{Fu84}, \cite{DG97}, \cite{E99}, \cite{Ma95}, \cite{Kol06}. 

One main example of applications of the theory has been the construction of processes which arise in certain problems of quantum mechanics with singular potentials (for which we refer to [KaStr14]) and quantum field theory (stochastic quantization equations for invariant measures and Hamiltonians for $P(\varphi)_2$, $\exp(\varphi)_2$, $\sin(\varphi)_2$, $\varphi_3^4$ and other models). Recently there arose a renewed interest in SPDE's related to such problems, particularly in connection with Hairer's theory of regularity structures \cite{Hai14} and related work by, a.a., Guminelli, Zambotti, see references in \cite{Hai14}.

Since the strength of results obtained by the theory of (quasi-regular) Dirichlet forms in connection with (infinite dimensional) processes has often not been fully realized in the literature, one of the aims of the present paper is to both recall main results and clarify both solved and not yet solved problems. \\
The structure of the paper is as follows. 

In Section 2 we shall recall the basic setting of the theory of symmetric quasi regular Dirichlet forms and associated processes, including invariant measures and ergodicity questions. 

Section 3 recalls the setting for classical infinite dimensional Dirichlet forms given by a probability measure, and its relation to infinite dimensional diffusion processes, their generators, associated Kolmogorov equations, and invariant probability measures. 

Section 4 is devoted to the study of the stochastic quantization equation for the $P(\varphi)_2$ model, with remarks on related models. 
We distinguish clearly between the problem in a bounded domain and the global problem. We also recall our result on the ergodicity of the solution.

Section 5 is devoted to the discussion of further developments concerning non symmetric Dirichlet forms. It is pointed out by two examples in infinite dimension that local invariance is indeed weaker than invariance of measures. On the way we also give a new presentation of an example of non-Markov uniqueness (in fact in the language of SPDE), first constructed in \cite{E99}.

%%%%%%%%%%%%%%%%%%%%%%%%%%%%%%%%%%%%%%%%%%%%%%%%%%%%%%%%%%%%%%%%%%%%%%%%%%%%%%%%%%%%%%%%%%%%%%%%%%%%
%%%%%%%%%%%%%%%%%%%%%%%%%%%%%%%%%%%%%%%%%%%%%%%%%%%%%%%%%%%%%%%%%%%%%%%%%%%%%%%%%%%%%%%%%%%%%%%%%%%%
%%%%%%%%%%%%%%%%%%%%%%%%%%%%%%%%%%%%%%%%%%%%%%%%%%%%%%%%%%%%%%%%%%%%%%%%%%%%%%%%%%%%%%%%%%%%%%%%%%%%

\section{Symmetric quasi regular Dirichlet forms}

We first recall some basic notations of the theory of symmetric quasi regular Dirichlet forms, for later use. 

Let $E$ be a Hausdorff topological space, $m$ a $\sigma$-finite measure on $E$, and let $\mathcal{B}$ the smallest $\sigma$-algebra of subsets of $E$ with respect to which all continuous functions on $E$ are measurable. 

Let $\mathcal{E}$ be a symmetric Dirichlet form acting in the real $L^2(m)$-space i.e.\ $\mathcal{E}$ is a positive, symmetric, bilinear, closed form with domain $D(\mathcal{E})$ dense in $L^2(m)$, and such that $\mathcal{E}(\Phi(u),\Phi(u)) \le \mathcal{E}(u,u)$, for any $u \in D(\mathcal{E})$, where $\Phi(t) = (0 \vee t) \wedge 1$, $t \in \mathbb{R}$. The latter condition is known to be equivalent with the condition that the associated $C_0$-contraction semigroup $T_t$, $t \geq 0$, is submarkovian (i.e. $0 \le u \le 1$ $m$-a.e.\ implies $0 \le T_t u \le 1$ $m$-a.e., for all $u \in L^2(m)$; association means that $\lim_{t \downarrow 0} \frac{1}{t} (u - T_tu,v)_{L^2(m)} = \mathcal{E}(u,v)$, $\forall u,v \in D(\mathcal{E})$. 

These conditions are also equivalent to other conditions expressed either in terms of $\mathcal{E}$ or the associated infinitesimal generator $L$ of $T_t$, $t \geq 0$, resp.\ the resolvent, see, e.g., \cite{Alb03} (Theorem 9). One observation which is important for an analytic construction is that for having the above contraction property in terms of $\Phi$ it is enough to verify it on a domain where the form is still closable and with $\Phi$ replaced by a regularized version $\Phi_{\varepsilon}$ of it (see, e.g., Definition 18 and Theorem 9 in \cite{Alb03}). \\
A symmetric Dirichlet form is called quasi-regular if the following holds:

\begin{enumerate}
	\item There exists a sequence $(F_k)_{k \in \mathbb{N}}$ of compact subsets of $E$ such that $\bigcup_k D(\mathcal{E})_{F_k}$ is $\mathcal{E}_1^{\frac{1}{2}}$-dense in $D(\mathcal{E})$ (where $D(\mathcal{E})_{F_k} := \{ u \in D(\mathcal{E}) \vert u=0 \ m\text{-a.e.\ on} \ E-F_k \}$; $\mathcal{E}_1^{\frac{1}{2}}$ is the norm given by the scalar product in $L^2(m)$ defined by $\mathcal{E}_1$, where
				\begin{equation}
					\mathcal{E}_1 (u,v) := \mathcal{E}(u,v) + (u,v)_{L^2(m)} \text{,} 
					\nonumber
				\end{equation}
				$(\, , \,)_{L^2(m)}$ being the scalar product in $L^2(m)$. Such a sequence $(F_k)_{k \in \mathbb{N}}$ is called an $\mathcal{E}$-nest;
	%----------------------------------------------------------------------------------------------------%
	\item	there exists an $\mathcal{E}_1^{\frac{1}{2}}$-dense subset of $D(\mathcal{E})$ whose elements have $\mathcal{E}$-quasi-everywhere $m$-versions (where ``quasi'' is relative to the potential theory associated with $\mathcal{E}$, i.e.\ quasi-everywhere (noted q.e.) means with the possible exception of some $\mathcal{E}$-exceptional subset of $E$, i.e.\ a subset $N \subset \bigcap_k (E - F_k)$ for some sequence $F_k$ of the above type; this is equivalent with $N$ having $\mathcal{E}_1$-capacity $0$, see \cite{Alb03}, \cite{MR92});
	%----------------------------------------------------------------------------------------------------%
	\item	there exists $u_n \in D(\mathcal{E})$, $n \in \mathbb{N}$, with $\mathcal{E}$-quasi continuous $m$-versions $\tilde{u}_n$ and there exists an $\mathcal{E}$-exceptional subset $N$ of $E$ s.t.\ $\{ \tilde{u}_n \}_{n \in \mathbb{N}}$ separates the points of $E-N$ (a real function $u$ on $E$ is called quasi continuous when there exists an $\mathcal{E}$-nest s.t.\ $u$ is continuous on each $F_k$, and is defined on a domain in $E$ containing $\bigcup_k F_k$; for this it is sufficient that given $\varepsilon > 0$ there exists an $U_{\varepsilon}$ open in $E$, s.t.\ cap $U_{\varepsilon} < \varepsilon$ and $u$ is continuous on $E - U_{\varepsilon}$ (\cite{MR92})).
\end{enumerate}

If $E$ is a locally compact separable metric space then $\mathcal{E}$ regular implies $\mathcal{E}$ quasi-regular (but not viceversa, in general). \\
The relation between these analytic notions and the probabilistic notions goes as follows. A submarkovian semigroup $p_t$ acting in $\mathcal{B}_b(E) \cap L^2(m)$ is associated with a symmetric Dirichlet form $(\mathcal{E}, D(\mathcal{E}))$ on $L^2(m)$ if $p_tu$ is an $m$-version of $T_tu$, where $T_t$, $t \geq 0$ is the $C_0$-contraction semigroup associated with $(\mathcal{E}, D(\mathcal{E}))$. 

A stochastic sub-Markov process $\boldsymbol{M} = (\Omega, \mathcal{F}, (\mathcal{F}_t)_{t \geq 0}, (X_t)_{t \geq 0}, (P_z)_{z \in E})$ with state space $E$ and transition semigroup $p_t$, in the sense that $(p_tu)(z) = E^zu (X_t)$, $\forall z \in E$, $t \geq 0$, $u \in \mathcal{B}_b(E) \cap L^2(m)$, is said to be associated with $\mathcal{E}$ or $(T_t, t \geq 0)$ if $p_t$ is associated with $\mathcal{E}$. $p_t$ is $m$-symmetric (i.e.\ a symmetric semigroup in $L^2(m)$) and, in the Markov case, i.e.\ when $(P_t1)(z) = 1$, has $m$ has an invariant measure. If $p_tu$ is $\mathcal{E}$-quasi-continuous for all $t > 0$ then $\boldsymbol{M}$ is said to be properly associated with $\mathcal{E}$ (this is a substitute of the strong Feller property in the present case, where $E$ is non necessarily locally compact). It turns out that

\begin{enumerate}
	\item $\mathcal{E}$ is quasi-regular iff $\boldsymbol{M}$ is an $m$-tight special standard process (in the sense of, e.g., \cite{MR92});
	%----------------------------------------------------------------------------------------------------%
	\item $\mathcal{E}$ is a local (in the sense that $\mathcal{E}(u,v) = 0$ for all $u,v \in D(\mathcal{E})$ with $\operatorname{supp}(\left| u \right| \cdot m) \cap \operatorname{supp}(\left| v \right| \cdot m) = \varnothing$ and $\operatorname{supp}(\left| u \right| \cdot m),\operatorname{supp}(\left| v \right| \cdot m)$ compact)\footnote{These supports are well defined since $E$ can be assumed, without loss of generality, to be a Lusin space, see \cite{MR92} (p.\ 101).} and quasi regular iff $\boldsymbol{M}$ is an $m$-tight special standard process and it is a diffusion i.e.\ $P_z \{ t \to X_t \ \text{continuous on} \ [0,\zeta) \} = 1$ for all $z \in E$ (for some random variable $\zeta$, with values in $[0, +\infty]$, the life time of $X$).
\end{enumerate}

A rough description of $m$-tight special standard processes is that they are right processes ``concentrated on compacts'', with almost-sure left limits and are almost-surely quasi left continuous. Right processes are adapted, strong Markov, normal (i.e.\ $P_z (X_0 = z) = 1$, for all $z = E \cup \{ \Delta \}$, $\{ \Delta \}$ being the $1$-point set describing the ``cemetery'' for the process), and have right continuous paths (see \cite{MR92}, \cite{Alb03} for details). \\
The connection of quasi regular Dirichlet forms with stochastic analysis goes beyond potential theory inasmuch as it permits an extension to the infinite dimensional case of Fukushima's decomposition.

% Theorem
\begin{thm}
	Let $\boldsymbol{M}$ be a right process associated with a quasi regular symmetric Dirichlet form $(\mathcal{E}, D(\mathcal{E}))$ on a Hausdorff topological space $E$. If $u \in D(\mathcal{E})$ then there exists a martingale additive functional of finite energy $M^{[u]}$ and a continuous additive functional of zero energy $N$ s.t.\ for any quasi-continuous version $\tilde{u}$ of $u$ we have
	\begin{equation}
		\tilde{u} (X_t) = \tilde{u} (X_0) + M_t^{[u]} + N_t^{[u]} \text{.}
		\nonumber
	\end{equation}
\end{thm}

For the proof of this theorem see \cite{Fu80}, \cite{FOT11} in the locally compact case and \cite{AMR92c}, \cite{MR92} (Chapter VI, Theorem 2.5) in the general case. \\
For the notions of martingale resp.\ continuous zero-energy additive functional see \cite{MR92}, \cite{Alb03} (roughly $E^z \left( \left( M_t^{[u]} \right)^2 \right) < \infty$, $E^z \left( M_t^{[u]} \right) = 0$, $\mathcal{E}$-q.e.\ $z \in E$, $\forall t \geq 0$; $M_t^{[u]}$ is a martingale under $P_x$ for any $x \in E - N$, with $N$ a properly exceptional set; $\lim_{t \downarrow 0} \frac{1}{2t} E^m \left( \left( N_t^{[u]} \right)^2 \right) = 0$, $E^z \left( \left| N_t^{[u]} \right| \right) < \infty$ q.e.\ $z \in E$). \\
Let us conclude this section with a short discussion of problems of stochastic dynamics. \\
Given a probability measure $\mu$ on some space $E$ one can ask the question whether there exists a Markov process $\boldsymbol{M}$ with Markov transition semigroup $p_t$, $t \geq 0$, $\mu$-symmetric (in the sense that the adjoint $p_t^{\ast}$ of $p_t$ in $L^2(\mu)$ coincides with $p_t$, for all $t \geq 0)$. In particular, then $\mu$ is $p_t$ invariant i.e.
%----------------------------------------------------------------------------------------------------%
\begin{equation}
	\int p_tu \, \mathrm{d}\mu = \int u \, \mathrm{d}\mu \text{,} \quad \forall u \in L^2 (\mu) \text{.}
	\label{eqn3}
\end{equation}
%----------------------------------------------------------------------------------------------------%
One calls then $p_t$ the ``stochastic dynamics'' associated with $\mu$. \\
A probability measure $\mu$ is said to be infinitesimal invariant under a $C_0$-semigroup $T_t$, $t \geq 0$, if $\int Lu \, \mathrm{d}\mu = 0$ for all $u$ in a subset of the domain $D(L)$ of $L$, which is dense in $L^2(E;\mu)$, where $L$ is the infinitesimal generator of $T_t$. 

In general, however, infinitesimal invariance of $\mu$ does not imply its invariance for the corresponding semigroup, i.e.\ does not imply (\ref{eqn3}). There are many counterexamples known (see e.g.\ \cite{BStR00} or \cite{BKRS14}). In Section \ref{sec:5} we shall give an explicit (Gaussian) counterexample with state space being a Hilbert space, in which even the underlying martingale problem is well-posed. \\
For references on invariant and infinitesimally invariant measures see, e.g., \cite{ABR99}, \cite{AFe04}, \cite{AFe08}, \cite{AKR97a}, \cite{AKR97b}, \cite{AKR98}, \cite{AKR02}, \cite{ARW}, \cite{AR02}, \cite{MR10}, \cite{MZ10}, \cite{ARZ93b}, \cite{BoRZh}, \cite{DPZa}, \cite{E99}, \cite{BKRS14}. \\
To the above ``inverse problem'' there exists a direct problem: given a (Markov) stochastic process $\boldsymbol{M}$, find (if possible) a probability measure $\mu$ s.t.\ $\mu$ is is an invariant measure for $\mu$ in the sense that (\ref{eqn3}) holds. 

We shall see in Section 4 how these problems are solved in the case of the processes associated with a special class of Dirichlet forms, namely classical Dirichlet forms (in the sense of \cite{AR90}, \cite{MR92}, \cite{Alb03}, \cite{AKR02}). 

A further general question related with stochastic dynamics concerns the uniqueness of the invariant measure for a given stochastic process $\boldsymbol{M}$, see, e.g., \cite{DPZa}, \cite{PeZ}, \cite{Eche} for results on this problem. \\
Let us now briefly recally the relation of Dirichlet forms with martingale problems, related uniqueness problems, and large time asymptotic behaviour of associated semigroups. \\
Let us consider a topological space $E$, a set $\mathcal{A}$ of real-valued functions on $E$ and a linear operator $L$ defined on $\mathcal{A}$. A probability measure $P$ on a probability space $\Omega$ consisting of continuous paths on $E$ with possible finite life time is said to be a solution of the martingale problem for $(L, \mathcal{A})$ relative to a coordinate stochastic process $X_t$, $t \geq 0$ on $E \cup \{ \Delta \}$ ($\Delta$ being the terminal point and $X_t (\omega) = \omega (t)$, $\omega \in \Omega$, with some initial distribution $\mu$) if for all $f \in \mathcal{A}$ and $t \geq 0$ we have that $f(X_t)$ and $Lf(X_t)$ are in $L^1(P)$, the function $(\omega,s) \to (Lf) (X_s(\omega))$ is in $L^1(P \otimes \mathrm{d}s)$ for all $0 \le s \le t$, and
%----------------------------------------------------------------------------------------------------%
\begin{equation}
	M_t^{[f]} := f(X_t) - \int_0^t (Lf) (X_s) \mathrm{d}s
	\nonumber
\end{equation}
%----------------------------------------------------------------------------------------------------%
is a martingale with the given initial probability measure $\mu$ with respect to the filtration generated by $X_s$, for all $0 \le t < \infty$. \\
For a general discussion of martingale problems see \cite{StV}, \cite{EK86}, \cite{MR92}, \cite{AR89a}, \cite{AR90b}, \cite{ARZ93b}, \cite{E99}, \cite{BhK93}. 

The Markov uniqueness problem concerns the question whether in the class of Markov processes $X_t$, $t \geq 0$ the solution of the martingale problem with generator $L$ acting in $L^2(\mu)$ is already determined by restricting the generator to a subset $\mathcal{A}$ strictly contained in the $L^2(\mu)$-domain of $L$, but still dense in $L^2(\mu)$. See \cite{AHK82}, \cite{E99}, \cite{AR90b}, \cite{BK95}, [AuR02], \cite{SS03}, \cite{Tak92}. An example that Markov uniqueness can fail to hold is given in Section \ref{sec:5}, which is originally from \cite{E99}, but which is presented here in an updated form. \\
The strong uniqueness problem concerns the question whether $L$ is already essentially self-adjoint on $\mathcal{A}$ in $L^2(\mu)$, see \cite{AKR02}, \cite{AKR92}, \cite{AKR93}, \cite{LR98}. \\
The large time asymptotics of processes associated with Dirichlet forms can be discussed in terms of properties of the associated Dirichlet forms. \\
We recall some basic results in this direction. \\
Let $(T_t)$, $t \geq 0$, be a submarkovian $C_0$-contraction semigroup in $L^2(m)$. $T_t$ is called irreducible if $T_t(uf) = uT_tf$, for all $f \in L^{\infty}(m)$ and all $t > 0$ implies $u = \text{constant}$ $m$-a.e.. A submarkovian $C_0$-contraction semigroup $T_t$ in $L^2(m)$ is called $L^2(m)$-ergodic if $T_tu \to \int u \, \mathrm{d}m$ as $t \to \infty$ in $L^2(m)$, for all $u \in L^2(m)$. 

Let $\mathcal{E}$ be a symmetric Dirichlet form associated with an $m$-symmetric submarkovian $C_0$-contraction semigroup $T_t$ in $L^2(m)$. Then the statement that $T_t$ is $L^2(m)$-ergodic is equivalent with $T_t$ being irreducible and is also equivalent with $\mathcal{E}$ irreducible. These properties are also equivalent with the statement that $u \in D(L)$ and $Lu= 0$ imply $u = \text{const}$ $m$-a.e., where $L$ is the infinitesimal generator of $T_t$. \\
For the proof of these results see \cite{AKR97a} (and also \cite{AHK76}, \cite{AKR97b}). \\
For the connection of the above properties with properties of the right process associated to $T_t$ resp.\ $\mathcal{E}$ via Dirichlet forms theory see Section 4.3 (and, for related problems, in the case of non local Dirichlet forms, e.g., [AR{\"u}02], [AR{\"u}05]).

%%%%%%%%%%%%%%%%%%%%%%%%%%%%%%%%%%%%%%%%%%%%%%%%%%%%%%%%%%%%%%%%%%%%%%%%%%%%%%%%%%%%%%%%%%%%%%%%%%%%
%%%%%%%%%%%%%%%%%%%%%%%%%%%%%%%%%%%%%%%%%%%%%%%%%%%%%%%%%%%%%%%%%%%%%%%%%%%%%%%%%%%%%%%%%%%%%%%%%%%%
%%%%%%%%%%%%%%%%%%%%%%%%%%%%%%%%%%%%%%%%%%%%%%%%%%%%%%%%%%%%%%%%%%%%%%%%%%%%%%%%%%%%%%%%%%%%%%%%%%%%

 \section{Classical Dirichlet forms on Banach spaces \newline and weak solutions to SDE}
 Let $E$ be a separable real Banach space with dual $E^{'}$ and dualization $\sideset{_{E^{'}}^{}}{_{E}^{}}{\mathop{\langle \;,\; \rangle}} \; $. Let $\mathcal{B}(E)$ denote its Borel $\sigma$-algebra and let $\mu$ be a finite positive
 measure on $(E,\mathcal{B}(E))$ with $\mathrm{supp}[\mu]=E$. Define for $K\subset E^{'}$ the linear spaces
 \begin{align*}
  \mathcal{F}C_b^\infty(K):=\{f(l_1, \dots , l_m)|m\in \mathbb{N}, f\in C^\infty_b(\mathbb{R}^m), l_1, \dots, l_m \in K\}.  \tag{3.1}
 \end{align*}
Set $\mathcal{F}C_b^\infty:=\mathcal{F}C_b^\infty(E^{'})$. Compared with the finite dimensional case we have that $E,\; \mu,\; \mathcal{F}C_b^\infty(K)$ replace $\mathbb{R}^d,\; \mathrm{d}x, \; C_0^\infty(\mathbb{R}^d)$
respectively, where $\mathrm{d}x=$\ Lebesgue measure. We want to define a gradient $\nabla$. To this end fix $u=f(l_1, \dots, l_m)\in \mathcal{F}C_b^\infty, z \in E$ and define for $k\in E$, $s \in \mathbb{R}$:
\begin{equation*}
 \frac{\partial u}{\partial k}(z):= \frac{\mathrm{d}}{\mathrm{d}s} u(z+sk)_{s=0}= \sum_{i=0}^m \frac{\partial f}{\partial x_i}(l_1(z), \dots, l_m(z)) ~ \sideset{_{E^{'}}^{}}{_{E}^{}}{\mathop{\langle l_i,k \rangle}}. \tag{3.2}
\end{equation*}
Furthermore, we assume that we are given a ``tangent space'' $H$ to $E$ at each point, in the sense that $H$ is a separable real Hilbert space such that $H\subset E$ continuously and densely. 
Thus, identifying $H$ with its dual $H^{'}$ we have
\begin{equation*}\tag{3.3}
 E^{'}\subset H\subset E \textnormal{  continuously and densely}. 
\end{equation*}
We define $\nabla u(z)$ to be the unique element in $H$ such that
\begin{equation*}
 \langle \nabla u(z), h\rangle_H = \frac{\partial u}{\partial h }(z) \textnormal{ for all } h\in H(\subset E).
\end{equation*}
Now it is possible to define a positive definite symmetric bilinear form (henceforth briefly called $form$) on (real) $L^2(E;\mu)$ by
\begin{equation*}\tag{3.4}\label{3.4}
 \mathcal{E}_{\mu}^0 (u,v):=\frac{1}{2} \int_E \langle \nabla u, \nabla v \rangle_H \mathrm{d} \mu ; \  \ u,v\in \mathcal{F}C^\infty_b,
\end{equation*}
which is densely defined, since $\mathcal{F}C^\infty_b$ is dense in $L^2(E;\mu)$ (cf. \cite[Chap. 4, Sect.b)]{MR92}). An element $k\in E$ is called \textit{well-$\mu$-admissible} if there exists $\beta_k^\mu \in L^2(E;\mu)$ such that
\begin{equation*}\tag{3.5}\label{3.5}
 \int_E u\;  \mathrm{d} \mu = - \int_E u \; \beta_k^\mu \; \mathrm{d} \mu \textnormal{  for all  } u \in \mathcal{F}C_b^\infty.
\end{equation*}
Assume:
\begin{gather*} \tag{3.6}\label{3.6}
  \textnormal{There exists a linear subspace } K \textnormal{ of } E^{'}(\subset H\subset E), \quad \qquad \qquad \qquad \; \; \\
 \textnormal{point separating on }E \textnormal{ and consisting} \textnormal{ of well-} \mu \textnormal{-admissible elements in }E.
\end{gather*}
which we fix from now on in this section. Then it is easy to see that the form $(\mathcal{E}_{\mu}^0, \mathcal{F}C^\infty_b)$ defined in \eqref{3.4} is closable on 
$L^2(E;\mu)$ and that its closure $(\mathcal{E}_{\mu}, D(\mathcal{E}_{\mu}))$ is a Dirichlet form on $L^2(E;\mu)$ (cf. \cite{AR90}, \cite[Chap.II, Subsection 3b]{MR92}).
We also denote the closure of $\nabla$ with domain $D(\mathcal{E}_{\mu})$ by $\nabla$, hence 
\begin{equation*} \tag{3.7} \label{3.7}
 \mathcal{E}_\mu (u,v) = \frac{1}{2}\int_E\langle \nabla u, \nabla v \rangle_H ~\mathrm{d} \mu; ~~ u,v\in D(\mathcal{E}_\mu).
\end{equation*}
The Dirichlet form in \eqref{3.7} is called \textit{classical (gradient) Dirichlet form} given by $\mu$ (see \cite{AR90}, \cite{MR92}).
Let $L_\mu$ with domain $D(L_\mu)$ be its generator (i.e. $\mathcal{E}_{\mu}(u,v) = (u, (-L_{\mu})v)$, with $(\, , \,)$ the scalar product in $L^2(E;\mu)$, $u \in D(\mathcal{E}_{\mu})$, $v \in D(L_{\mu})$) (cf. \cite[Chap. I]{MR92}). $(L_\mu, D(L_\mu))$ is a \textit{Dirichlet operator}, i.e., $(e^{tL_\mu})_{t\ge0}$ is sub-Markovian.
It is immediate that if $u=f(l_1,\dots,l_m) \in \mathcal{F}C_b^\infty(K) \textnormal{ and } K_0\subset K$ is an orthonormal
basis of $H$ having $l_1,\dots,l_m$ in its linear span, then $u\in D(L_\mu)$ and
\begin{equation*}\tag{3.8}\label{3.8}
 L_\mu u = \frac{1}{2} \sum_{k\in K_0} (\frac{\partial}{\partial k}(\frac{\partial u}{\partial k})+\beta_k^\mu \frac{\partial u}{\partial k}).
\end{equation*}

\begin{thm}\label{3.1}
 The Dirichlet form $(\mathcal{E}_\mu, D(\mathcal{E}_\mu))$ defined in \eqref{3.7} is local and quasi-regular.
\end{thm}

\begin{proof}
See \cite[Chap. 5, Example 1.12 (ii)]{MR92} for the locality and for the quasi-regularity see \cite{RS92} and also  \cite[Chap. 4, Sect. 4b]{MR92}.
\end{proof}

Hence by Section 2 there exists a (Markov) diffusion process $\textbf{M}_\mu=(\Omega, \mathcal{F}, (\mathcal{F}_t), (X_t), (P_z)_{z\in E})$ \textit{ properly associated with} 
the Dirichlet form $\mathcal{E}_\mu, D(\mathcal{E}_\mu))$, i.e., for all $u\in L^2(E;\mu), ~ t>0,$
\begin{equation*}\tag{3.9}\label{3.9}
 \int u(X_t) \; \mathrm{d} P_z = T_t u(z) \textnormal{ for } \mu - \textnormal{a.e. } z \in E
\end{equation*}
where $T_tu:=\exp(tL_\mu)u$, and the function on the left hand side of \eqref{3.9} has an $\mathcal{E}_\mu$-quasi continuous version. 
$\textbf{M}_\mu$ is easily seen to be conservative (i.e., has infinite lifetime) and to have $\mu$ as an invariant measure.
By \cite[Chap. 4, Sect. 4b]{MR92} it follows that there is a point separating countable $\mathbb{Q}$-vector space $D\subset \mathcal{F}C_b^\infty(K)$.
It is easy to see that there exists an ONB $K_0 \subset K$ of $H$ such that the linear span of $K_0$ separates the points in $E$.
Let us make for what follows a fix choice of such sets $K_0$ and $D$. Then as an immediate consequence of (a special case of)
the Fukushima decomposition (see Sect.2) (in particular we use the version of \cite{MR92}, Chapter VI, Theorem 2.5), we obtain the following result concerning the martingale problem for $(L_\mu,D)$:

\begin{thm}\label{3.2}
 There exists a set $S \in \mathcal{B}(E)$ such that $E\backslash S$ is \textbf{properly $\mathcal{E}_\mu$-exceptional}, i.e. $\mu(E\backslash S)=0$ and $P_z[X_t \in S \; \forall t\ge 0]=1 \textnormal{ for all }z\in S$,
 and for all $u\in D$ (see above) 
 \begin{equation*}
  u(X_t)-\int^t_0 L_\mu u(X_s)\; \mathrm{d} s, ~~ t\ge 0,
 \end{equation*}
is an $(\mathcal{F}_t)$-martingale under $P_z$ for all $z\in S$. 
\end{thm}

As a consequence of Theorem 3.2 and the results in \cite{AR91} we obtain that $\textbf{M}_\mu$ yields weak solutions to the corresponding stochastic differential equation (SDE) on $E$. 
More precisely, we have:

\begin{thm}
 Let S be as in Theorem 3.2 and assume that there exists a $\mathcal{B}(E)/\mathcal{B}(E)$-measurable map $\beta^\mu:E\longrightarrow E$ such that
 \begin{itemize}
  \item[(i)] $\sideset{_{E^{'}}^{}}{_{E}^{}}{\mathop{\langle k, \beta^\mu \rangle}}= \beta^{\mu}_k ~~\mu$-a.e for each $k\in K$;
  \item[(ii)] $\int_E || \beta^\mu ||^2_E\; \mathrm{d} \mu < \infty$.
 \end{itemize}
Then there exists an $E$-valued $(\mathcal{F}_t)_{t\ge0}$-Wiener process $W_t, ~t\ge0, \textnormal{ on } (\Omega, \mathcal{F},P)$ such that 
$(\sideset{_{E^{'}}^{}}{_{E}^{}}{\mathop{\langle k, W_t \rangle}})_{k \in K_0}$ is a cylindrical Wiener process in $H$ and for every $z\in S$
\begin{equation*}
 X_t = z + \frac{1}{2} \int_0^t \beta ^\mu (X_s) \; \mathrm{d} s + W_t , ~~t\ge0,
\end{equation*}
$P_z$-a.s. (where the integral is an $E$-valued Bochner integral). 
\end{thm}

\begin{proof}
 \cite[Theorem 6.10]{AR91}.
\end{proof}

%%%%%%%%%%%%%%%%%%%%%%%%%%%%%%%%%%%%%%%%%%%%%%%%%%%%%%%%%%%%%%%%%%%%%%%%%%%%%%%%%%%%%%%%%%%%%%%%%%%%
%%%%%%%%%%%%%%%%%%%%%%%%%%%%%%%%%%%%%%%%%%%%%%%%%%%%%%%%%%%%%%%%%%%%%%%%%%%%%%%%%%%%%%%%%%%%%%%%%%%%
%%%%%%%%%%%%%%%%%%%%%%%%%%%%%%%%%%%%%%%%%%%%%%%%%%%%%%%%%%%%%%%%%%%%%%%%%%%%%%%%%%%%%%%%%%%%%%%%%%%%

\section{Applications to stochastic quantization \newline in finite and infinite volume}

Quantum field theory has its origin in physics (by work by Born, Dirac, Heisenberg, Jordan, Pauli, a.a., see, e.g., \cite{Jo}) as an attempt to quantize the classical theory of relativistic fields in a similar way as non relativistic quantum mechanics is the quantization of classical (non relativistic) field theory, see, e.g., [ASe14]. \\
In contrast to non relativistic quantum mechanics, a mathematical sound construction of the dynamics encountered difficulties which are still not overcome. A systematic attempt to perform such a construction was initiated in the late 60's and culminated in the early 70's with the construction of models describing interactions, in the case of an idealized $d$-dimensional space-time world (for $d=1,2,3$), see \cite{S74}, \cite{GJ81}, \cite{AFHKL86}. For the physically most interesting case $d=4$ only partial results have been obtained, see \cite{AGW97}, \cite{AGY05}. \\
In these approaches a non Gaussian probability measure $\mu$ on a distributional space (e.g.\ $\mathcal{S}'(\mathbb{R}^d)$) is constructed having the heuristic expression:
\begin{equation}
	\mu^* (\mathrm{d}\varphi) = Z^{-1} e^{- \int_{\mathbb{R}^d} v(\varphi(x)) \mathrm{d}x} \mu_0^{\ast} (\mathrm{d}\varphi) \text{,} \quad \text{with} \ Z = \int_{\mathcal{S}' (\mathbb{R}^d)} e^{- \int_{\mathbb{R}^d} v (\varphi(x)) \mathrm{d}x} \mu_0^{\ast} (\mathrm{d}\varphi) \text{,}
	\label{eqn1}
\end{equation}
$\varphi \in \mathcal{S}' (\mathbb{R}^d)$ being a symbol for a space-time process connected with the physical quantum field in $d$-space-time dimensions. $v$ is a real-valued function, describing the interaction. $\mu_0^{\ast}$ is a Gaussian random field, with mean $0$ and covariance $E(\varphi(x) \varphi(y)) = (-\Delta + m^2)^{-1} (x,y)$, the fundamental solution or Green function of $-\Delta + m^2$, $m$ being a positive constant (in physical terms $m$ is the mass of particles described by the free field, below we shall always take for simplicity $m=1$). \\
For $v\equiv0$ we have heuristically $\mu^* = \mu_0^{\ast}$. A particularly well studied case for $v\not\equiv0$ is where $v(y) = P(y)$, $y \in \mathbb{R}^d$, with $P$ a polynomial with a positive even highest order term. $\mu^{\ast}$ is then called ``Euclidean measure for the $P(\varphi)_d$-field'' (for more general $v$ $\mu^{\ast}$ yields the ``Euclidean measure for the $v(\varphi)_d$-field''). \\
The rigorous construction of $\mu^{\ast}$ uses tools of probability theory and statistical mechanics, see \cite{S74}, \cite{GRS75}, \cite{GJ81}, \cite{AFHKL86}. \\
The analogy of the heuristic formula (\ref{eqn1}) with the (canonical, Gibbs) equilibrium measure in statistical mechanics makes it natural to ask, both from physics and mathematics, whether it is possible to find a Markov diffusion process ($X_t$, $t \geq 0$) such that (\ref{eqn1}) appears as equilibrium measure (invariant measure) for this process. A formal computation can be performed to see that a solution $X_t$ of the following heuristic stochastic quantization equation
\begin{equation}
	\mathrm{d}X_t = (\Delta - 1) X_t \, \mathrm{d}t - v'(X_t) \mathrm{d}t + \mathrm{d}W_t, \; t\ge0,
	\label{eqn2}
\end{equation}
would have $\mu^*$ as an (heuristic) invariant measure, where $\mathrm{d}W_t$ stands for Gaussian white noise in all variables $(t,x)$, $x \in \mathbb{R}^d$. In the physical interpretation, $t$ here is a ``computer time'', whereas the Euclidean space-time-variables $x\in \mathbb{R}^d$ appear in
$X_t=X_t(x)$. Heuristically, $X_0(x) = \varphi(x)$, with $\varphi$ as in (\ref{eqn1}). $\Delta=\Delta_x$ is the Laplacian in $\mathbb{R}^d$.
The physical reason for asking the above question is to exploit the dynamics of solutions of such an equation to perform Monte-Carlo simulations of physically relevant averages with respect to the equilibrium measure $\mu^*$ (see \cite{PaWu81}, and, e.g., \cite{Mit89}). The mathematical interest arises because of the ``typicality'' of an equation of the form (\ref{eqn2}) for handling SPDE's with singular noise. Due to the singular character of the measure $\mu^*$ (whenever it exists!) one expects (for $d \geq 2$) troubles in giving a meaning to the term $v'(X_t)$ in (\ref{eqn2}). In this section we shall see that for $d=2$ this programme can be rigorously achieved (the case $d=1$ is simpler, since for $d=1$, $x \to X_t(x)$, $x \in \mathbb{R}$, is continuous). For our discussion we shall separate 2 cases: 
the ``finite volume case'' (Section 4.1), where $\mathbb{R}^2$ is replaced by a rectangle $\Lambda$ and the ``infinite volume case'' (Section 4.2), where the ``space cut-off'' $\Lambda$ is eliminated and the problem is considered in the whole space $\mathbb{R}^2$.

%%%%%%%%%%%%%%%%%%%%%%%%%%%%%%%%%%%%%%%%%%%%%%%%%%%%%%%%%%%%%%%%%%%%%%%%%%%%%%%%%%%%%%%%%%%%%%%%%%%%
%%%%%%%%%%%%%%%%%%%%%%%%%%%%%%%%%%%%%%%%%%%%%%%%%%%%%%%%%%%%%%%%%%%%%%%%%%%%%%%%%%%%%%%%%%%%%%%%%%%%
%%%%%%%%%%%%%%%%%%%%%%%%%%%%%%%%%%%%%%%%%%%%%%%%%%%%%%%%%%%%%%%%%%%%%%%%%%%%%%%%%%%%%%%%%%%%%%%%%%%%

\subsection{Finite volume}
Let $\Lambda$ be an open rectangle in $\mathbb{R}^2$. Let $(-\Delta+1)_N$ be the generator of the following quadratic form on 
$L^2(\Lambda ; \mathrm{d} x):(u,v) \longrightarrow\int_\Lambda \langle \nabla u, \nabla v \rangle_{\mathbb{R}^d}\; \mathrm{d} x + \int_\Lambda u\;v\; \mathrm{d}x$ with $u, v \in \{g\in L^2 (\Lambda; \mathrm{d} x)| \nabla g \in L^2(\Lambda; \mathrm{d} x)\}$
(where $\nabla$ is in the sense of distributions, $N$ reminds us to ``Neumann boundary conditions''). Let $\{e_n | n\in \mathbb{N} \} \subset C^\infty(\bar \Lambda)$ be the (orthonormal) eigenbasis of $(-\Delta + 1)_N$ and
$\{\lambda_n |\in \mathbb{N} ~~(\subset ]0,\infty[)$ the corresponding eigenvalues (cf. \cite[p. 266]{RS78}). Define for $\alpha \in \mathbb{R}$ 
\begin{equation*}\tag{4.1}
 H_\alpha := \{ u \in L^2(\Lambda; \mathrm{d} x)| \sum_{n=1}^\infty \lambda ^\alpha _n\; \langle u, e_n\rangle^2_{L^2(\Lambda ; \mathrm{d} x)} < \infty \},
\end{equation*}
equipped with the inner product
\begin{equation*} \tag{4.2} \label{eqn4.2}
 \langle u, v \rangle_{H_\alpha} := \sum_{n=1}^\infty \lambda^\alpha_n \; \langle u, e_n \rangle_{L^2(\Lambda; \mathrm{d} x)}\langle v,e_n\rangle_{L^2(\Lambda;\mathrm{d} x )}.
\end{equation*}
Clearly, we have that 
\begin{equation*}\tag{4.3}
H_\alpha =\begin{cases}
           \text{completion of } C^\infty(\bar \Lambda) \text{ w.r.t. } \| \; \|_{H^\alpha} \text{ if } \alpha < 0\\
           \text{completion of } C_0^\infty(\Lambda) \text{ w.r.t. } \| \; \|_{H_\alpha} \text{ if } \alpha \ge 0
          \end{cases}
 \end{equation*}
(cf. \cite[p. 79]{LM72} for the latter). \\ To get into the framework of Section 3 we chose
\begin{equation*}
 H:=L^2(\Lambda; \mathrm{d} x), ~~~ E:=H_{-\delta}, \ \delta>0.
\end{equation*}
Then
\begin{equation*}\tag{4.4}\label{4}
 E^{'}( ~=H_\delta)~~\subset H ~~ \subset E.
\end{equation*}
\begin{rem}
 In \eqref{4} we have realized the dual of $H_{-\delta}$ as $H_\delta$ using as usual the chain 
 \begin{equation*}
  H_\alpha ~~\subset L^2(\Lambda ; \mathrm{d} x)~~\subset H_{-\alpha}, ~~~~\alpha \ge 0.
 \end{equation*}
 \end{rem}
Fix $\delta >0$. Since $\sum_{n=1}^\infty \lambda_n^{-1-\delta} < \infty$ (cf.\ [RS75]), we have, applying \cite[Theorem 3.2]{Y89} (i.e., the Gross-Minlos-Sazonov theorem) with
$H:=L^2(\Lambda ; \mathrm{d} x ), ~||\cdot|| := || \cdot||_{H_{-\delta}}, ~ A_1:= (-\Delta +1)_N^{-\frac{\delta}{2}}$, and $A_2:=(-\Delta +1)_N^{-\frac{1}{2}}$, 
that there exists a unique mean zero Gaussian probability measure $\mu$ on $E:=H_{-\delta}$ (called \textit{free field} on $\Lambda$ with Neumann boundary conditions; see \cite{N73})
such that 
\begin{equation*}\tag{4.5}\label{4.4}
 \int_E \sideset{_{E^{'}}^{}}{_{E}^{}}{\mathop{\langle l, z \rangle^2}} \;\mu (\mathrm{d}z)=||l||^2_{H_{-1}} \textnormal{ for all } l\in E^{'} = H_\delta.
\end{equation*}
Clearly, $\mathrm{supp}\;\mu = E$. For $h\in H_{-1}$ we define $X_h\in L^2(E;\mu)$ by
\begin{equation*}\tag{4.6}\label{4.5}
 X_h:= \lim_{n\rightarrow \infty} \sideset{_{E^{'}}^{}}{_{E}^{}}{\mathop{\langle k_n, \cdot \rangle}}~~\textnormal{ in } L^2(E;\mu),
\end{equation*}
where $(k_n)_{n\in \mathbb{N}}$ is any sequence in $E^{'}$ such that $k_n \underset{n\rightarrow \infty}\longrightarrow h$ in $H_{-1}$.

Let $h \in \mathrm{Dom}((-\Delta+1)_N) ~~(\subset L^2(\Lambda;\mathrm{d} x) \subset E)$, then 
\[\sideset{_{E^{'}}^{}}{_{E}^{}}{\mathop{\langle k, h \rangle}}=\langle k,h \rangle_{L^2(\Lambda, \mathrm{d} x)} = \langle k, (-\Delta+1)_N h \rangle_{H_{-1}}\]
for each $k \in E^{'}$. Hence by \cite[Proposition 5.5]{AR90}
\eqref{4.4}, and \eqref{4.5}, $h$ is well-$\mu$-admissible and $\beta^\mu_h=X_{(-\Delta+1)_Nh}$ (see \eqref{3.5}). 
 Let $K$ be the linear span of $\{e_n|n\in \mathbb{N}\}$. Below we shall consider the gradient Dirichlet form $(\mathcal{E}_{\bar \mu}, D(\mathcal{E}_{\bar \mu}))$ on $L^2(E;\bar \mu)$
 as introduced in Section 3, where $\bar \mu$ is the $P(\varphi)_2$-quantum field measure in the finite volume $\Lambda$ with $P$ being a polynomial, see below just before Proposition \ref{prop4.2}, for its mathematical description.\\
 To this end we first have to introduce the so-called Wick powers $:z^n:,~n\in \mathbb{N}$, which are renormalized powers of the Schwartz-Sobolev distribution $z\in E=H_{-\delta}$.\\
 Let $h\in L^2(\Lambda;\mathrm{d} x), ~ n\in \mathbb{N}$, and define $:z^n:(h)$ as a limit in $L^p(E;\mu), ~p\in \; [1,\infty[$, as follows (cf., e.g., \cite[Sect. 8.5]{GJ81}): fix $n \in \mathbb{N}$ 
 and let $H_n(t),~t\in \mathbb{R}$, be the $n$th Hermite polinomial, i.e., $H_n(t) := \sum_{m=0}^{[n/2]} (-1)^m \alpha_{nm} t^{n-2m}$, 
 with\\ $\alpha_{nm} := n!/[(n-2m)!2^mm!]$. Let $d \in C_0^\infty(\mathbb{R}^2), ~ d\ge 0, ~\int d(x) \mathrm{d} x =1$, and $d(x)=d(-x)$ for each $x\in \mathbb{R}^2$. Define for 
 $k \in \mathbb{N}, ~d_{k,x}(y):=2^{2k}d(2^k(x-y)); ~~x,y \in \mathbb{R}^2$. Let $z_k(x):= \sideset{_{E^{'}}^{}}{_{E}^{}}{\mathop{\langle d_{k,x}, z \rangle}}\;, \;z\in E, x\in \Lambda$, and set
 \begin{equation*}\tag{4.7}
  :z_k^n:(x):=c_k(x)^{n/2}H_n(c_k(x)^{-1/2}z_k(x)),
 \end{equation*}
where $c_k(x):=\int z_k(x)^2\mu (\mathrm{d}z)$. Then it is known that 
\[:z^n_k:(h):=\int:z_k^n:(x) h(x) \mathrm{d} x \underset{k\rightarrow \infty}\longrightarrow:z^n:(h)\]both in every 
$L^p(E;\mu), ~ +\infty>p\ge1$, and for $\mu$-a.e. $z\in E$ (cf.,e.g., \cite[Sect. 3]{R86} for the latter and \cite{S74} for the former). The map $z\mapsto \limsup_{k\rightarrow\infty} :z^n_k:(h)$ is then a $\mu$-version of $:z^n:(h)$. 
From now on $:z^n:(h)$ shall denote this particular version. Since $H_n(s+t)=\sum_{m=0}^n \binom{n}{m} H_m(s)t^{n-m}, ~s,t \in \mathbb{R}$, we also have that if 
\[z\in M:=\{z\in E|\limsup_{k\rightarrow \infty} :z_k^n(h): \;=\lim_{k\rightarrow \infty} :z_k^n:(h) (\in \mathbb{R})\}\] then $\mu(M)=1$, $z+k\in M$ for all $k\in K$ and
\begin{equation*}\tag{4.8}
 :(z+k)^n:(h) = \sum_{m=0}^n \binom{n}{m} :z^m:(k^{n-m} h)
\end{equation*}
(cf. \cite[Sect. 3]{R86} for details).\\
Now fix $N\in \mathbb{N}, a_n \in \mathbb{R}, ~ 0\leq n\leq2N$ with $a_{2N}>0$ and define 
\begin{equation*}
 V(z):=:P(z):(1_\Lambda):=\sum_{n=0}^{2N} a_n :z^n:(1_\Lambda), ~~~z\in E,
\end{equation*}
where $1_\Lambda$ denotes the indicator function of $\Lambda$. Let 
\begin{equation*}\tag{4.9}
\varphi :=\exp (-\frac{1}{2} V).
\end{equation*}                                                                                                                                                                                                 
Then $\varphi >0$ $\mu$-a.e. and $\varphi \in L^p(E;\mu)$ for all $p \in \; [1, \infty[$ (cf., e.g., \cite[Sect. 5.2]{S74} or \cite[Sect. 8.6]{GJ81}).
Set
\[\bar \mu := \varphi^2 \cdot \mu.\]
\begin{prop}\label{prop4.2}
 Each $k \in K$ is well-$\mu$-admissible and 
 \begin{equation*}\tag{4.10}\label{4.10}
  \beta^{\bar \mu}_k=-X_{[(-\Delta +1)_Nk]}-\sum_{n=1}^{2N} na_n :z^{n-1}:(k)
 \end{equation*}
(cf. \eqref{3.5} above for the definition of $\beta^{\bar \mu}_k$).
\end{prop}
\begin{proof}
 \cite[Proposition 7.2]{RZ92}. 
\end{proof}
It now follows that Theorem  \ref{3.1} applies to the corresponding classical (gradient) Dirichlet form $(\mathcal{E}_{\bar \mu}, D(\mathcal{E}_{\bar \mu}))$
introduced in Section 3.

So, let $\textbf{M}_{\bar \mu}= (\Omega, \mathcal{F}, (\mathcal{F}_t), (X_t)_{t\ge 0}, (P_z)_{z\in E})$ be the corresponding (Markov) diffusion process as in Section 3.
Then Theorem \ref{3.2} applies, i.e. we have solved the martingale problem (in the sense of Theorem \ref{3.2}) for the corresponding operator $L_{\bar \mu}$ given by 
\eqref{3.8} with $\beta_k^{\bar \mu}$ replacing $\beta_k^{\mu}$. Finally, taking $\delta$ for $E=H_{-\delta}$ large enough, there exists $\beta^{\bar \mu}:E\rightarrow E,~ \mathcal{B}(E)/\mathcal{B}(E)$ measurable and 
satisfying conditions \textit{(i)} and \textit{(ii)} in Theorem 3.3 (see \cite[Proposition 6.9]{AR91} ). 
By construction and \eqref{4.10} we have that 
\begin{equation*}\tag{4.11}
 \beta^{\bar \mu}(z)= -(-\Delta+1)_N~ z\;-\; :P^{'}(z):, ~ z\in E.
\end{equation*}
Hence Theorem 3.3 implies the existence of a set $S \in \mathcal{B}(E)$ such that $\bar \mu(E\backslash S)=0$ and $S$ is invariant under the (Markov) diffusion process 
$(X_{t})_{t\ge 0}$, i.e. $P_{z}[X_{t} \in S\; \forall t\ge 0]=1$
for all $z\in S$, and for every $z\in S$
\begin{equation*}\tag{4.12}\label{4.12}
 X_t= z \; -\; \frac{1}{2} \int_0^t\; [(-\Delta +1)_N X_s\; +\; :P^{'}(X_s):] \mathrm{d} s \; + \; W_t, ~~t\ge 0,
\end{equation*}
$P_z$-a.e., for some $L^2(\Lambda^2; \mathrm{d} x)$-cylindrical $(\mathcal{F}_t)$-Wiener process on $(\Omega, \mathcal{F}, P_z)$, i.e. we have
a Markov weak solution to the SDE \eqref{4.12} (``weak'' in the probabilistic sense). 
The solution to \eqref{4.12} is usually called \textit{stochastic quantization process} for the $P(\varphi)_2$-quantum field with Neumann boundary condition in the finite volume $\Lambda$ (and (\ref{4.12}) is a rigorous ``finite volume'' version of the heuristic stochastic quantization equation (\ref{eqn4.2})).

%%%%%%%%%%%%%%%%%%%%%%%%%%%%%%%%%%%%%%%%%%%%%%%%%%%%%%%%%%%%%%%%%%%%%%%%%%%%%%%%%%%%%%%%%%%%%%%%%%%%
%%%%%%%%%%%%%%%%%%%%%%%%%%%%%%%%%%%%%%%%%%%%%%%%%%%%%%%%%%%%%%%%%%%%%%%%%%%%%%%%%%%%%%%%%%%%%%%%%%%%
%%%%%%%%%%%%%%%%%%%%%%%%%%%%%%%%%%%%%%%%%%%%%%%%%%%%%%%%%%%%%%%%%%%%%%%%%%%%%%%%%%%%%%%%%%%%%%%%%%%%

\subsection{Infinite volume}
For $n\in \mathbb{Z}$ define the space $\mathcal{S}_n$ to be the completion of $C_0^\infty(\mathbb{R}^2)$($=$ all compactly supported smooth
functions on $\mathbb{R}^2$) with respect to the norm
\begin{equation*}
 \|k\|_n:=\Big[ \sum_{|\underline m|\leq n} \int_{\mathbb{R}}(1+|x|^2)^n\Big|\left(\frac{\partial^{m_1}}{\partial x_1^{m_1}} \frac{\partial^{m_2}}{\partial x_2^{m_2}}\right)k(x)\Big|^2 \mathrm{d} x\Big],
\end{equation*}
where $\underline m :=(m_1,m_2) \in (\mathbb{Z}_+)^2$. To get into the framework of Section 3 we choose 
\[K:=C_0^\infty(\mathbb{R}^2), \; H:=L^2(\mathbb{R}^2;\mathrm{d}x),\; E:=\mathcal{S}_{-n},\]
where $n \in \mathbb{N}$, which later we shall choose large enough. So, we have 
\[E^{'}\; \subset H\; \subset E.\]
Let $\mu_0^*$ be the space time free field of mass $1$ on $\mathbb{R}^2$, i.e. $\mu_0^*$ is the unique centered Gaussian measure on $E$ with covariance operator
$(-\Delta + 1)^{-1}$. For $h\in L^{1+\varepsilon}(\mathbb{R}^2;\mathrm{d} x)$, with $\varepsilon>0$, and $n\in \mathbb{N}$, let $:z^n:(h)$ be defined analogously as in Section 3, but with $\mu_0^*$ taking the role of $\mu$ ($=$ the free field of mass $1$ on $\Lambda$
with Neumann boundary condition).
From now on we fix $N\in \mathbb{N}, \; a_n \in \mathbb{R}, \; 0\leq n\leq2N$, and define for $h\in L^{1+\varepsilon}(\mathbb{R}^2;\mathrm{d}x)$

\begin{equation*}\tag{4.13}
 :P(z):(h):=\sum_{n=0}^{2N} a_n:z^n:(h)\; \textnormal{ with }\; a_{2N}>0.
\end{equation*}
We have that $:P(z):(h), \; \exp(-:P(z):(h))\in L^p(E; \mu_0^*)$ for all $p\in [1,\infty[$, if $h\ge 0$ (cf. e.g.\cite[$\S$V.2]{S74}), 
hence the following probability measures (called \textit{space-time cut-off quantum fields}) are well-defined for $\Lambda \in \mathcal{B}(\mathbb{R}^2), \Lambda$ bounded,
\begin{equation*}\tag{4.14}
 \mu^*_\Lambda :=\frac{\exp(-:P(z):(1_\Lambda))}{\int \exp(-:P(z):(1_\Lambda))\mathrm{d} \mu_0^*}\cdot \mu_0^*.
\end{equation*}
It has been proven that the weak limit
\begin{equation*}\tag{4.15}
\lim_{\Lambda \nearrow \mathbb{R}^2} \mu^*_\Lambda=: \; \mu^* 
\end{equation*}
exists as a probability measure on $(E,\mathcal{B}(E))$ (see \cite{GJ81} and also \cite{R86}, \cite{R88}) having moments of all orders. Furthermore, by \cite[Proposition 2.7]{AR90b} we have that
$\mathrm{supp} \mu^* = E$ (i.e. $\mu^*(U)>0$ for each open subset $U$ of $E$). It is, however, well-known that $\mu^*$ is not absolutely continuous with respect to $\mu^*_0$ (since $\mu^{\ast} \not= \mu_0^{\ast}$ is invariant under translations in the underlying $\mathbb{R}^2$-space, which act ergodically, see \cite{AHK74}, \cite{S74}, \cite{GRS75}, \cite{Fro}, \cite{AL}). Now we have:
\begin{prop}
 Each $k\in K$ is well $\mu^*$-admissible with 
 \begin{equation*}\tag{4.16}\label{4.16}
  \beta_k^{\mu^*}(z):= \;- \sum_{n=1}^{2N} na_n :z^{n-1}:(k) \; - \; \sideset{_{E^{'}}^{}}{_{E}^{}}{\mathop{\langle (-\Delta +1)k, z \rangle}}, \; z\in E.
 \end{equation*}
\end{prop}
\begin{proof}
 \cite[Theorem 7.11]{AR91}.
\end{proof}
As a consequence of Proposition 4.2, according to Section 3 we obtain the corresponding classical (gradient) Dirichlet form $(\mathcal{E}_{\mu^*},D(\mathcal{E}_{\mu^*}))$, which is quasi regular
by Theorem 3.1. So, let $\textbf{M}_{\mu^*}=(\Omega, \mathcal{F}, (\mathcal{F}_t), (X_t)_{t\ge 0}, (P_z)_{z\in E})$
be the coresponding (Markov) diffusion process as in Section 3. Then Theorem 3.2 applies, i.e. we have solved the martingale problem (in the sense of Theorem 3.2) for the corresponding operator
$L_{\mu^*}$ given by \eqref{3.8} with $\beta^{\mu^*}_k$, as given by (\ref{4.16}), replacing $\beta^\mu_k$. 

Finally, taking $n$ large enough, there exists $\beta^{\mu^*} :E\rightarrow E,\; \mathcal{B}(E)/\mathcal{B}(E)$-measurable and satisfying condition \textit{(i)} and \textit{(ii)} in Theorem 3.3 (see \cite[Proposition 6.9]{AR91}).
By construction and \eqref{4.16} we have that
\begin{equation*}\tag{4.17}
 \beta^{\mu^*}(z)= (\Delta -1)z \;-\; :P^{'}(z):, \; z\in E.
\end{equation*}
Hence Theorem 3.3 implies the existence of a set $S\in \mathcal{B}(E)$ such that $\mu^*(E\backslash S)=0$ and $S$ is invariant under the (Markov) diffusion process $(X_t)_{t\ge 0}$, i.e.
$P_z[X_t\in S \; \forall t\ge 0]=1$ for all $z\in S$, and, for every $z\in S$, $X_t$ solves the stochastic integral equation
\begin{equation*}\tag{4.18}\label{4.18}
X_t = z+ \frac{1}{2} \int_0^t \Big[(\Delta -1) X_s \; -\; :P^{'}(X_s):\Big]\mathrm{d} s \;+\; W_t,\; t\ge 0,
 \end{equation*}
$P_z$-a.s. for some $L^2(\mathbb{R}^2; \mathrm{d} x)$-cylindrical $(\mathcal{F}_t)$-Wiener process on $(\Omega, \mathcal{F}, P_z)$, i.e. we have a Markov weak solution to the SDE \eqref{4.18}
(``weak'' in the probabilistic sense). The solution to \eqref{4.18} is usually called \textit{stochastic quantization process} for the $P(\varphi)_2$-quantum field in infinite volume (thus, with $\Lambda$ in Section 4.1 replaced by $\mathbb{R}^2$).

\begin{rem}
 All the above results also hold for the ``time zero quantum fields'' associated with the $P(\varphi)_2$-quantum field in infinite volume, first discussed in \cite{AHK74}. For details on this we refer to \cite[Section 7, II b)]{AR91}.
\end{rem}

%%%%%%%%%%%%%%%%%%%%%%%%%%%%%%%%%%%%%%%%%%%%%%%%%%%%%%%%%%%%%%%%%%%%%%%%%%%%%%%%%%%%%%%%%%%%%%%%%%%%
%%%%%%%%%%%%%%%%%%%%%%%%%%%%%%%%%%%%%%%%%%%%%%%%%%%%%%%%%%%%%%%%%%%%%%%%%%%%%%%%%%%%%%%%%%%%%%%%%%%%
%%%%%%%%%%%%%%%%%%%%%%%%%%%%%%%%%%%%%%%%%%%%%%%%%%%%%%%%%%%%%%%%%%%%%%%%%%%%%%%%%%%%%%%%%%%%%%%%%%%%

\subsection{Ergodicity}
The use of Dirichlet form techniques is not limited to settling existence questions for solutions of SDE's with very singular coefficients.
Also important special properties of solutions can be deduced. As one instance we consider the situation of the previous subsection, and ask about ergodic 
properties of the solution.

To start, we first mention that the construction of the infinite volume $P(\varphi)_2$-quantum field $\mu^*$, which we took as a reference
measure in the previous section, is quite specific. There is, in fact, quite a large set of possible reference measures that could 
replace $\mu^*$, namely all so-called \textit{Guerra-Rosen-Simon $P(\varphi)_2$-quantum fields} (see \cite{GRS75}), which are defined as the convex 
set $\mathcal{G}$ of all Gibbs measures for a certain specification, i.e. they are defined through the classical Dobrushin-Landford-Ruelle
equations appropriate to $P(\varphi)_2$-quantum fields. We do not go into details here and do not give the precise definition of the relevant specification, but rather refer to \cite{R86} 
(see also \cite[Section 4.3]{AKR97b}). We only recall from \cite{R86} that each $\nu$ from the convex set $\mathcal{G}$ can be represented as an integral over the set $\mathcal G_{ex}$ of 
all extreme points of $\mathcal{G}$. Furthermore, we recall that by the main result in \cite{AR89a} (see also \cite{AKR97b}) for every $\nu\in\mathcal{G}$ 
the corresponding form \eqref{3.4}, with $\nu$ replacing $\mu$, is closable on $L^2(E;\nu)$, so its closure $(\mathcal{E}_\nu,D(\mathcal{E}_\nu))$ 
(see \eqref{3.7}) is a classical (gradient) Dirichlet form for which all results from Section 4.2 apply with $\nu$ replacing $\mu^*$.
Then the following is a special case of \cite[Theorem 4.14]{AKR97b}.
\begin{thm}\label{4.4}
Suppose $\nu\in\mathcal{G}_{ex}$. Then $(\mathcal{E}_\nu,D(\mathcal{E}_\nu))$ is irreducible and (equivalently) the corresponding (Markov) diffusion 
process $\textbf{M}_\nu$ is time ergodic under \[P_\nu :=\int_E P_z \; \nu(\mathrm{d}z).\]
\end{thm}
\begin{rem}\label{4.5} \mbox{}
\begin{itemize}
\item [(i)] In fact, if we replace $(\mathcal{E}_\nu,D(\mathcal{E}_\nu))$ by the corresponding ``maximal Dirichlet form'' 
$(\mathcal{E}_\nu^{max},D(\mathcal{E}_\nu^{max}))$ (see \cite{AKR97b} for the definition), then the irreducibility of the latter implies 
that $\nu\in\mathcal{G}$ must be in $\mathcal{G}_{ex}$.
\item [(ii)] The equivalence of the irreducibility of $(\mathcal{E}_\nu,D(\mathcal{E}_\nu))$ and the time ergodicity of $\textbf{M}_\nu$ under $P_\nu$ follows 
by the general theory (see \cite{F82}).In this case the irreducibility of $(\mathcal{E}_\nu,D(\mathcal{E}_\nu))$  is also equivalent to the $L^2$-ergodicity of the 
corresponding semigroup $(T_t^\nu)$, $t\ge 0$, i.e.\[\lim\limits_{t\rightarrow \infty}\|T_t^\nu f-f\|_{L^2(E;\nu)}=0~~~\textnormal{ \textit{for all} }~~f\in L^2(E;\nu).\]
For details we refer to \cite[Proposition 2.3]{AKR97a}. The latter paper was dedicated to Professor Masatoshi Fukushima on the occasion of his 60th birthday. A comprehensive study 
of the above relations between irreducibility of a Dirichlet form, the time ergodicity of the corresponding Markov process and the $L^2$-ergodicity of the corresponding semigroup and  other 
related properties is contained in the forthcoming paper \cite{BCR14} in a much more general context, including non-symmetric coercive and generalized Dirichlet forms.
\item[(iii)] Theorem \ref{4.4}, stated above for Guerra-Rosen-Simon Gibbs states of $P(\phi)_2$-Euclidean quantum field theory, is valid for many other 
Gibbs states of both lattice and continuum systems from statistical mechanics. We refer to \cite[Section 5]{AKR97a}, \cite{AKKR09} for lattice systems and to \cite[Section 7]{AKR98} for continuum 
systems.
\end{itemize}
\end{rem}

%%%%%%%%%%%%%%%%%%%%%%%%%%%%%%%%%%%%%%%%%%%%%%%%%%%%%%%%%%%%%%%%%%%%%%%%%%%%%%%%%%%%%%%%%%%%%%%%%%%%
%%%%%%%%%%%%%%%%%%%%%%%%%%%%%%%%%%%%%%%%%%%%%%%%%%%%%%%%%%%%%%%%%%%%%%%%%%%%%%%%%%%%%%%%%%%%%%%%%%%%
%%%%%%%%%%%%%%%%%%%%%%%%%%%%%%%%%%%%%%%%%%%%%%%%%%%%%%%%%%%%%%%%%%%%%%%%%%%%%%%%%%%%%%%%%%%%%%%%%%%%

\subsection{Additional remarks}

\begin{enumerate}
	\item Above we discussed the stochastic quantization equation as it was first proposed in \cite{PaWu81}, where the driving noise is $\mathrm{d}W_t$ i.e.\ a Gaussian space-time white noise. The same invariant measure can also be obtained by considering an SPDE with a space-regularized noise, in which case the drift term has to be modified accordingly. E.g., in the $\varphi_2^4$-case the regularized stochastic quantization equation has been discussed in the finite volume case, with various types of regularizations, in \cite{DPT00}, \cite{HK98}, \cite{JLM85}, \cite{LR98}, \cite{GaGo}, \cite{DPD}.
	
	In particular in this finite volume case a proof of Markov uniqueness in $L^p(\mu)$, $1 \le p < \infty$ has been achieved in \cite{RZ92}. 
				Strong solutions have been constructed in suitable Besov spaces \cite{DPD}, and essential self-adjointness of generators has been proven in \cite{LR98} and \cite{DPT00}. 
The invalidity of a Girsanov formula has been shown in \cite{MiRo}. 
The discussion of uniqueness and ergodicity questions we sketched in Section 4.3 carry over to this case too. \\
The only papers, where the infinite volume processes which solve the stochastic quantization equation, in their original or regularized version, have been discussed, seem to be \cite{AR89b}, \cite{AR91}, \cite{BCM88}. The uniqueness results which hold in the finite volume case do not carry over to the infinite volume case, where the question of uniqueness of generators is still open. See, however, \cite{ARZ93a}, \cite{AR89a} for some partial results. 
	\item	The regularized stochastic quantization equation has also been discussed for other $2$ space-time dimensional models. The starting point for the construction of solutions by the Dirichlet form method, namely the measure described heuristically given by (\ref{eqn1}) resp.\ its analogue in a bounded volume, has been constructed rigorously for the case, where e.g.\ $v$ is an exponential function $v(y) = \exp(\alpha y)$ or a superposition of such functions, of the form $v(y) = \int \exp(\alpha y) \nu_v (\mathrm{d}y)$, with $\nu_v$ a bounded variation measure with $\left| \alpha \right| < \sqrt{4 \pi}$ resp.\ $\operatorname{supp} \nu_v \in (-\sqrt{4 \pi}, \sqrt{4 \pi})$, in \cite{AHK74}, (see also \cite{HK71}, \cite{S74}, \cite{FrP77}, \cite{AGHK79}, \cite{AFHKL86}, and \cite{AHPRS90a}, \cite{AHPRS90b}, \cite{HKPS93} for an alternative construction using methods of white noise analysis). \\
				The corresponding regularized stochastic quantization equation has been discussed in a bounded domain in \cite{Mi} and \cite{AKMR}. Let us point out that models of this type are presently under intensive investigation in regard to their importance in completely different areas of research, see, e.g., \cite{RhVa13}. \\
				For models given in terms of other functions $v$, e.g., superposition of trigonometric functions, like the function $v$ appearing in the $\sin \alpha \varphi_2$-model (the quantized version of the Sine-Gordon equation in $2$ space-time dimensions) constructions of the measure heuristically given by (\ref{eqn1}) as well as the definition of the dynamics in terms of Dirichlet forms have also been discussed, see \cite{AHR01}. In this work the necessity of renormalization has been shown and strong solutions have been constructed in a suitable distributional setting See also \cite{AHPRS90a}, \cite{AHPRS90b} for a white noise analysis approach to such $v$'s.
	\item	Recent work has concerned both the stochastic quantization equation in space-time dimension $d=1$ and $d=3$. \\
For $d=1$ the dynamics has been constructed in the strong probabilistic sense and $L^p$-uniqueness of the generators has been proven, see \cite{Iwa85}, \cite{Iwa87} resp.\ \cite{KR07}, for $v$ of polynomial type, resp. \cite{AKaR12}, for $v$ exponential or trigonometric type, both in finite and infinite volume. \\
In the case $d=3$ an integration by parts formula has been established for the $\varphi_3^4$-model, as well as the existence of a pre-Dirichlet form \cite{ALZ06}. 
However, the (generalized) logarithmic derivative does not seem to have good enough integrability properties. Therefore, the closability of this pre-Dirichlet form is an open problem, as is the existence of a global (Markov) dynamics. 
This might be put in relation with a recent approach to the local stochastic dynamics developed for this model by M.\ Hairer \cite{Hai14}.
\end{enumerate}

%%%%%%%%%%%%%%%%%%%%%%%%%%%%%%%%%%%%%%%%%%%%%%%%%%%%%%%%%%%%%%%%%%%%%%%%%%%%%%%%%%%%%%%%%%%%%%%%%%%%
%%%%%%%%%%%%%%%%%%%%%%%%%%%%%%%%%%%%%%%%%%%%%%%%%%%%%%%%%%%%%%%%%%%%%%%%%%%%%%%%%%%%%%%%%%%%%%%%%%%%
%%%%%%%%%%%%%%%%%%%%%%%%%%%%%%%%%%%%%%%%%%%%%%%%%%%%%%%%%%%%%%%%%%%%%%%%%%%%%%%%%%%%%%%%%%%%%%%%%%%%

\section{Further developments}
\label{sec:5}

%%%%%%%%%%%%%%%%%%%%%%%%%%%%%%%%%%%%%%%%%%%%%%%%%%%%%%%%%%%%%%%%%%%%%%%%%%%%%%%%%%%%%%%%%%%%%%%%%%%%
%%%%%%%%%%%%%%%%%%%%%%%%%%%%%%%%%%%%%%%%%%%%%%%%%%%%%%%%%%%%%%%%%%%%%%%%%%%%%%%%%%%%%%%%%%%%%%%%%%%%
%%%%%%%%%%%%%%%%%%%%%%%%%%%%%%%%%%%%%%%%%%%%%%%%%%%%%%%%%%%%%%%%%%%%%%%%%%%%%%%%%%%%%%%%%%%%%%%%%%%%

In applying Dirichlet form techniques to SDEs the symmetry assumption is, of course, very restrictive. Unfortunately, also the sector condition imposed 
on the (in general non-symmetric) sectorial Dirichlet forms analyzed in \cite{MR92} is still too restrictive to cover important classes of SDE, in particular on infinite
dimensional state spaces, including stochastic partial differential equations. Therefore, a theory of ``fully non-symmetric'' Dirichlet forms was developed
in \cite{Sta99}. The main feature that this ``theory of generalized Dirichlet forms'' has in common with the symmetric and sectorial case is that it still requires
a ``reference measure'' $\mu$ to be given or constructed beforehand. Below, we want to briefly describe an important subclass thereof,
where one is just given an operator $L$ and a measure $\mu$, intrinsically related to $L$ (see Definition 5.1 below),  that will serve as a ''reference measure''.
The underlying idea has been first put forward systematically in \cite{ABR99} and \cite{R98}, and was one motivation that 
has eventually led to the recent monograph \cite{BKRS14}.\\

Consider the situation described at the beginning of Section 3, so $E$ is a (real) separable Banach space, $H$ a (real) separable Hilbert space such that
$H\subset E$ continuously and densely, hence 
\[E'\subset~(H'\equiv)~H\subset~E,\] continuously and densely. Fix an algebra $K~\subset~E',$ containing a 
countable subset, separating the points of $E,$ thus $K$ generates $\mathcal{B}(E),$ i.e.
\begin{equation*}\tag{5.1}\label{5.1}
\sigma(K)=\sigma(E')=\mathcal{B}(E).
\end{equation*}
Let $L$ be a linear operator, whose domain contains $\mathcal{F}C_b^\infty(K)$ (cf. (\ref{3.1})).
\begin{defi}\label{def:5.1}
 A probability measure $\mu$ on $\mathcal{B}(E)$ is called L-infinitesimally invariant if $Lu\in L^1(E;\mu)$ for all $u~\in \mathcal{F}C_b^\infty(K)$ and 
 \begin{equation*}
  \int\limits_{E}Lu \; \mathrm{d}\mu=0 \text{ for all } u \in \mathcal{F}C_b^\infty(K),
 \end{equation*}
in short: if 
\begin{equation*}\tag{5.2}\label{5.2}
 L^*\mu=0.
\end{equation*}
\end{defi}

We note that for $\mu$ as in Definition \ref{def:5.1}, because of (\ref{5.1}), we have that $\mathcal{F}C_b^\infty(K)$ is dense in $L^p(E;\mu)$ for all
$p\in [1,\infty).$ Let us assume that $L$ is an operator of type (\ref{3.8}) or more generally a Kolmogorov or a diffusion operator (in the sense of \cite{E99}).
In this case quite general theorems are known to ensure the existence of measures $\mu$ satisfying (\ref{5.2}) (see \cite{BKRS14} and the references therein). It then
follows immediately, that $L$ is dissipative on $L^1(E;\mu)$ (see \cite[Lemma 1.8]{E99}), hence closable on $L^1(E;\mu)$. Let us denote its closure by $(L_\mu,D(L_\mu))$.
It is, however, not true in general, that $L_\mu$ generates a $C_0$-(contraction) semigroup $T_t^\mu=e^{tL_\mu}$, $t\ge 0$, on $L^1(E;\mu)$. If it does, 
then this semigroup is always sub-Markovian (see \cite{E99}). In this case one can ask whether there exists a (right) continuous Markov process on $E$ whose transition 
semigroup is related to $T_t^\mu$, $t\ge 0$, as in (\ref{3.9}). So, in summary, one can realize the ``Dirichlet form approach'' for such general Kolmogorov
operators $L$ to construct a corresponding (right) continuous Markov process in this ``fully non-symmetric'' case, if the following three problems can be solved:

\begin{itemize}
 \item [(1)] Does there exist a probability measure $\mu$ on $(E,\mathcal{B}(E))$  having the property (\ref{5.2}) in Definition 5.1?
 \item [(2)] Does $(L_\mu,D(L_\mu))$ then generate a $C_0$-(contraction) semigroup $T_t^\mu$, $t\ge 0$, on $L^1(E;\mu)$?
 \item[(3)] Does $T_t^\mu$, $t\ge 0$, come from a transition function $p_t, \; t\ge 0$, of a (right) continuous Markov?
 \end{itemize}
 
We note that if the answer to (2) is yes, then $\mu$ is also invariant for $T_t^\mu$, $t\ge 0$, i.e. 
\begin{equation*}\tag{5.3}\label{5.3}
 \int_{E}T_t^\mu u \; \mathrm{d}\mu=\int_{E} u \; \mathrm{d}\mu
\end{equation*}
for all $u\in\mathcal{F}C_b^\infty(K)$, equivalently for all $u\in L^1(E;\mu)$.
A lot of work has been done on the three problems above in the past decade.
We have already mentioned \cite{BKRS14} as a good reference for problem (1), but also problem (2) is discussed there and further references are given in \cite{BKRS14} 
concerning both (1) and (2). We only mention here that the answer ``yes'' to problem (2) is well-known to be equivalent to the ``range condition'', i.e. that 
$(1-L)(\mathcal{F}C_b^\infty(K))$ is dense in $L^1(E;\mu)$ (see e.g. \cite[Proposition 2.6]{BStR00}). But in concrete cases, in particular in infinite dimensions, this is a very
hard problem (see, however, \cite{Sta99b} for a useful characterization of the range condition if dim $E<\infty$). Concerning (3) considerable progress has
been made in \cite{BeBoR06}, \cite{BeBoR08}.

In the remainder of this subsection we want to address two important points concerning the above discussion, namely: \\
\newline
\textbf{Questions:} 
\begin{itemize}
 \item [(A)] Suppose that all problems (1)-(3) above can be solved for some $L$-infinitesimally invariant measure $\mu:=\mu_1$, which is hence invariant for $(p_t)_{t\ge0}.$
 Is it possible that there exists some other $L$-infinitesimally invariant measure $\mu_2$ which is not invariant for $(p_t)_{t\ge0}$?
 \item [(B)] Does there exist an $L$-infinitesimally invariant measure $\mu$, so that the closure of $(L, \mathcal{F}C^\infty_b(K))$ in $L^2(E;\mu)$ does not generate a $C_0$-semigroup,
 but there exist two different closed extentions generating two Markov $C_0$-semigroups on $L^2(E;\mu)$? In short: can ``Markov uniqueness'' fail to hold?
\end{itemize}

Indeed, the answer to (A) is yes.
There are well-known counterexamples $p_t, t\ge0$,  if $E=\mathbb{R}^d$ (see \cite{BKRS14} and the references therein, in particular, \cite{Sta99b}). Below we give two simple examples when $E$ is
an infinite dimensional Hilbert space, $\mu_2$ is a Gaussian measure and $(p_t)_{t\ge0}$ is even a strong Feller transition semigroup of an
Ornstein-Uhlenbeck process, i.e. is given explicitely by a Mehler-type formula. Also the answer to Question (B) is ''yes''. In fact, $\mu$ can even be chosen Gaussian and
so that $(L, \mathcal{F}C_b(K))$ is symmetric in $L^2(E;\mu)$ and both Markov $C_0$-semigroups consist of self-adjoint operators in $L^2(E;\mu)$. We shall, however, not give all details here, but refer instead to \cite{E99}
where these can be found (see Remark 5.2 below). We rather concentrate on details for the examples to answer Question (A) by ''yes''. \\

\textbf{First example:}

Consider the open interval $(0,1)\subset \mathbb{R}$ and choose
\begin{equation*}\tag{5.4}\label{5.4}
 H=E:=L^2\Big((0,1);\mathrm{d}x\Big),
\end{equation*}
with the usual inner product $\langle \;, \; \rangle$ and where $\mathrm{d}x$ denotes Lebesgue measure. Define $A_1:=-\Delta_D$, where $\Delta_D$ denotes
the Dirichlet Laplacian on $(0,1)$, and $A_2:=-\Delta_{D,N}$, where $\Delta_{D,N}$ denotes the Dirichlet-Neumann-Laplacian on $(0,1)$. 
More precisely, we take Dirichlet boundary condition at $\xi=0$ and Neumann boundary condition at $\xi=1$.
Define for $t\ge 0$, $x,y\in E$,
\begin{equation*}\tag{5.5}\label{5.5}
 p_t^{(i)}(x,\mathrm{d}y)=N(e^{-tA_i}x,Q_t^{(i)})(\mathrm{d}y),~~~~~i=1,2~~,
\end{equation*}
where $N$ denotes the Gaussian measure on $E$ with mean $e^{-tA_i}x$ and covariance operator
\begin{equation*}\tag{5.6}\label{5.6}
 Q_t^{(i)}:=\int\limits_{0}^{t}e^{-2sA_i}\mathrm{d}s~.
\end{equation*}
We note that, because we are on the one-dimensional underlying domain $(0,1)$, each $Q_t^i$ is indeed trace class. It is well-known that $p_t^{(i)}$,
$t\ge 0$, $i=1,2$, is the transition semigroup of the Ornstein-Uhlenbeck process $(X_t^{(i)})_{t\ge 0}$ solving in the mild sense the following SDE for $i=1,2$ respectively
\begin{align*}\tag{5.7}\label{5.7}
 \mathrm{d}X_t^{(i)}&=-A_iX_t^{(i)}\mathrm{d}t+\mathrm{d}W_t,~~~~~t\ge 0,\\
 X_0&=x,
\end{align*}
where $W_t$, $t\ge 0$, is an $H$-cylindrical Wiener process, and that 
\begin{equation*}\tag{5.8}\label{5.8}
 \mu_i:=N(0,\frac{1}{2}A_i^{-1}),~~~i=1,2~~,
\end{equation*}
is its respective invariant measure. 
Furthermore, it is well-known (and easy to see) that, for each $i\in \{1,2\}$, $p_t^{(i)}$, $t\ge 0$,
gives rise to a Markov $C_0$-semigroup $T_t^{\mu_i}$, $t\ge 0$, of symmetric contractions on $L^p(E;\mu_i)$ for all $p\in[1,\infty)$. 
Furthermore, for every bounded $\mathcal{B}(E)$-measurable function $f :E\rightarrow\mathbb{R}$
\begin{equation*}\tag{5.9}\label{5.9}
 p_t^{(i)}f(x):=\int\limits_E f(z)p_t^{(i)}(x,\mathrm{d}z),~~~x\in E~,
\end{equation*}
is continuous in $x\in E$, i.e. both $p_t^{(1)}$, $t\ge 0$, and $p_t^{(2)}$, $t\ge 0$, are strong Feller (see \cite{DPZa}).

Let us now consider $K:=C_0^2((0,1))$ and recall that as seen before, $\mathcal{F}C_b^\infty(K)$ is dense in $L^p(E,\mu_i)$ for all $p\in [1,\infty)$,
$i=1,2$. Then it is easy to check that for $i=1,2,\; p^{(i)}_t, t\ge0$, extends to a Markov $C_0$-semigroup $T^{\mu_i}_t, t\ge0, \textnormal{ on } L^2(E;\mu_i)$ and for the generator $(L_{\mu_i},D(L_{\mu_i}))$ of $T_t^{\mu_i}$, $t\ge 0$, on $L^2(E;\mu_i)$(but also poinwise, see \cite[Theorem 5.3]{BoR95}), 
 we have that $\mathcal{F}C_b^\infty(K)\subset D(L_{\mu_i})$ and for all
$u\in \mathcal{F}C_b^\infty(K)$, $u=F(\langle k_1,\cdot \rangle, \ldots,\langle k_N,\cdot \rangle)$, $F \in C_b^{\infty}(\mathbb{R}^N)$, $i=1,2$,
\begin{align*}\tag{5.10}\label{5.10}
L_{\mu_i}u(z)=\sum\limits_{j,j'=1}^{N}\langle k_j,k_{j'}\rangle \partial_{jj'}F(\langle k_1,z\rangle,\ldots,\langle k_N,z\rangle)\\
+\sum\limits_{j=1}^N\langle A_ik_j,z\rangle \partial_jF(\langle k_1,z\rangle,\ldots,\langle k_N,z\rangle),
\end{align*}
where $\partial_j,\partial_{j,j'}$ mean partial derivatives in the $j$-th or in the $j$-th and $j'$-th variable respectivly. Since $A_1k=A_2k$ 
for all $k\in K$, it follows that 
\begin{equation*}\tag{5.11}\label{5.11}
L_{\mu_1}u=L_{\mu_2}u \end{equation*} for all $u\in\mathcal{F}C_b^\infty(K)$.
Because of (\ref{5.11}) we may define 
\begin{equation*} \tag{5.12} \label{5.12}  Lu:=L_{\mu_1} u \;(=L_{\mu_2}  u), \; u\in \mathcal{F}C_b^\infty(K). \end{equation*}
Then because $T_t^{\mu_i}, \; t\ge0$, are symmetric in $L^2(E_i;\mu_i), \; i=1,2$, it follows that for $i=1,2$
\begin{equation*}\tag{5.13}\label{5.13}
 \int Lu \; v \; \mathrm{d} \mu_i = \int u \; Lv \;\mathrm{d}\mu_i \; \; \forall u,v \in \mathcal{F}C_b^\infty(K),
\end{equation*}
hence choosing $v\equiv1$
\begin{equation*}\tag{5.14}
 \int Lu \; \mathrm{d} \mu_i =0 \; \; \forall u\in\mathcal{F}C_b^\infty(K), 
\end{equation*}
i.e. $\mu_1$ and $\mu_2$ are $L$-infinitesimally invariant (for $L$ with domain $\mathcal{F}C_b^\infty(K)$).

\begin{rem}\mbox{}
\begin{itemize} 
 \item [(i)] We have just seen an infinite dimensional example where ``Markov uniqueness'' fails, i.e. two generators of two different Markov
 $C_0$-(contraction) semigroups which coincide on a common domain $\mathcal{F}C_b^\infty(K)$ which is dense in $L^2(E;\mu_i)$, $i=1,2$. 
 And both semigroups are even strong Feller in this case. We note, however, that in contrast to the ``classical'' Markov uniquess problem, our two $C_0$-semigroups live on different $L^2$-spaces, namely 
 $L^2(E;\mu_1)$ and $L^2(E;\mu_2)$. However, one can show that if one considers the Friedrichs extension $(L_{\mu_2,F}, D(L_{\mu_2, F}))$ of the (by (\ref{5.13}))
 symmetric operator $(L, \mathcal{F}C_b^\infty(K))$ on $L^2(E; \mu_2)$, then the corresponding Dirichlet form does not coincide with the Dirichlet form corresponding to the symmetric 
 Markov $C_0$-semigroup $(T_t^{\mu_2})_{t\ge0}$ introduced above (see \cite[Chap. 5b)1)]{E99}). Hence 
 \[\left( T_t^{\mu_2} \right)_{t\ge0}\neq\left( e^{tL_{\mu_2,F}}\right)_{t\ge0},\]
 and both are symmetric Markov $C_0$-semigroups on $L^2(E;\mu_2)$ with generators coinciding on $\mathcal{F}C_b^\infty(K)$, with $L$ defined in (\ref{5.12}).
 This is a ``true'' counter-example to Markov-uniqueness (first discorered in \cite{E99}) even for symmetric $C_0$-semigroups and (at least) one of the two semigroups is even strong Feller. 
 We stress that $\mathcal{F}C_b^\infty(K)$ is, of course, not an operator core for any of the generators of $(T_t^{\mu_2})_{t\ge0}$ and $(e^{tL_{\mu_2,F}})_{t\ge0}$ on $L^2(E;\mu).$
 \item[(ii)] Clearly, the above only occurs, because $\mathcal{F}C_b^\infty(K)$ is too small to determine $L_{\mu_2}$, since it does 
 not capture the boundary behaviour of $A_2$, since $K\subset C_0^2\Big((0,1)\Big)$, though $\mathcal{F}C_b^\infty(K)$ is dense in $L^2(E;\mu_2)$. For examples of non-Markov uniqueness avoiding this, we refer to \cite[Chap. 5b)2)]{E99}.
\end{itemize}
\end{rem}

It is easy to see that $\mathcal{F}C_b^\infty(K)$ is an operator core for $(L_{\mu_1},D(L_{\mu_1}))$ on $L^2(E;\mu_1)$, i.e. $\mathcal{F}C_b^\infty(K)$
is dense in $D(L_{\mu_1})$ with respect to the graph norm $\| \cdot \|_1:=\| \cdot \|_{L^2(E;\mu_1)} \; +\; \|L_{\mu_1} \cdot \|_{L^2(E;\mu_1)}$. Indeed, consider the Sobolev
space $H_0^1:= H^1_0 \Big((0,1);\mathrm{d} x\Big)$ of order $1$ in $L^2((0,1);\mathrm{d}x)$ with Dirichlet boundary conditions. Then it is obvious from \eqref{5.10} that $\mathcal{F}C_b^\infty(H^1_0)$
is in the closure of $\mathcal{F} C_b^\infty(K)$ with respect to the graph norm $\| \cdot \|_1$.
Furthermore, it is also obvious from the definition of $p_t^{(1)}, \; t\ge0$, in \eqref{5.5}, that $p_t^{(1)}(\mathcal{F}C_b^\infty(H_0^1))\; \subset \mathcal{F}C_b^\infty(H_0^1)$ for all $t\ge0$, since $e^{-tA_1}(H_0^1) \; \subset H^1_0$ for all $t\ge0$.
Hence by a theorem of Nelson (see \cite[Theorem X.49]{RS78}) if follows that $\mathcal{F}C_b^\infty(H^1_0)$, hence $\mathcal{F}C_b^\infty(K)$ is dense in $D(L_{\mu_1})$ with respect to $\| \cdot \|_1$. 

We recall that $\mu_2$ is $L$-infinitesimally invariant (for the domain $\mathcal{F}C_b^\infty(K)$). Now we show that $\mu_2$ is, however, not invariant for  $p_t^{(1)}, \; t\ge 0$,
hence giving the desired example for Question (A).
This can be proved as follows.
Fix $y\in E\backslash \{0\}$ and consider the function 
\begin{equation*}\tag{5.15}\label{5.15}
 f(z):= e^{i\langle y, z \rangle}, \; z \in E.
\end{equation*}
Then 
\[ \int_E f(z) \; \mu_2(\mathrm{d} z)=e^{- \frac{1}{2} \langle \frac{1}{2}A_2^{-1}y,\; y \rangle}\]
and
\begin{equation*}\tag{5.16}\label{5.16}  p_t^{(1)} f(z) = e^{i\langle e^{-tA_1}y,z \rangle}e^{-\frac{1}{2}\langle Q_t^{(1)}y, \; y \rangle},\end{equation*}
hence 
\[\int_E p_t^{(1)} f(z) \;\mu_2(\mathrm{d} z)= e^{-\frac{1}{2}\langle \frac{1}{2}A_2^{-1} e^{-2tA_1}y,\;y \rangle} \cdot e^{-\frac{1}{2} \langle Q_t^{(1)}y, \; y \rangle }.\]
But $Q_t^{(1)} = \frac{1}{2} A_1^{-1} (1-e^{-2tA_1})$, so in general (just choose $y$ in \eqref{5.15} to be an eigenvector of $A_1$)
\begin{equation*}\tag{5.17}\label{5.17}
\int_E f(z) \; \mu_2(\mathrm{d} z)\neq \int_E p_t^{(1)} f(z) \;\mu_2(\mathrm{d} z)
\end{equation*}

So, we have proved that $\mu_2$ is not invariant for $p_t^{(1)}$, $t\ge 0$.
\begin{rem}
Clearly, the martingale problem for $(L, \mathcal{F}C_b^\infty (K))$ is not well-posed, because the laws of the solutions $X^{(1)}$ and $X^{(2)}$ of \eqref{5.7} are solutions to this martingale problem, but do not coincide. The reason is
explained in Remark 5.2 (ii). It is, however, interesting to note that since, as explained above, $\mathcal{F}C_b^\infty(K)$ is an operator core for $(L_{\mu_1}, D(L_{\mu_1}))$ in $L^2(E;\mu_1)$,
it can be easily shown, that there exists at most one Markov selection for the martingale problem for $(L, \mathcal{F}C_b^\infty(K))$ such that the corresponding transition semigroup $(p_t)_{t\ge0}$ extends to a $C_0$-semigroup on $L^2(E;\mu_1)$.
\end{rem}
\textbf{Second Example:}

Define the measure $\mu$ by
\begin{equation*}\tag{5.18}\label{5.18}
 \mu:= N(1, \frac{1}{2} A^{-1}_1),
\end{equation*}
i.e. the image of $\mu_1$, defined in \eqref{5.8}, under the translation 
\[E\ni z \mapsto z+1 \in E,\]
with $1$ being the constant function equal to one on $(0,1)$, we obtain a Gaussian (not centered) measure on $E$ such that 
\begin{equation*}\tag{5.19}\label{5.19}
 L^*\mu = 0
\end{equation*}
and $\mu$ is not invariant for $p^{(1)}_t, t\ge0$. Indeed, \eqref{5.19} follows immediately from the definition of $L$ (by \eqref{5.12}, \eqref{5.10}) since
\[\langle A_1k, 1\rangle =  \int_0^1 -\Delta k \; \mathrm{d}x = 0 \textnormal{  for all } k\in K.\]
On the other hand, we have for $f$ as in \eqref{5.15} 
\[ \int_E f(z)\; \mu(\mathrm{d}z) = e^{i\langle y,1 \rangle}e^{-\frac{1}{2} \langle \frac{1}{2} A_1^{-1}y,y \rangle} \;,\]
but by \eqref{5.16}
\begin{align*}
 \int_E p^{(1)}_t f(z) \; \mu(\mathrm{d}z)\; &= \int_E p^{(1)}_t f(1+z) \; \mu_1(\mathrm{d}z)\\
 &= e^{i\langle e^{-tA_1}y,1\rangle} \int_E p_t^{(1)} f(z) \; \mu_1(\mathrm{d}z)\\
 &= e^{i\langle e^{-tA_1}y,1\rangle} \int_E f(z) \; \mu_1(\mathrm{d}z)\\
 &= e^{i\langle y, e^{-tA_1}1-1\rangle} \int_E f(z) \; \mu (\mathrm{d}z).\tag{5.20}
\end{align*}
But, since $A_1:= -\Delta_D$, we know that \[e^{-tA_1}1\neq1.\]
Hence $\mu$ is not invariant for $p^{(1)}_t, t\ge0$. 

\begin{rem} \mbox{}
     It is very easy to check that in the above example $p_t^{(1)}$, $t \geq 0$, is symme\-tric with respect to $\mu_1$ and that $(L, \mathcal{F} C_b^{\infty} (K))$ is symmetric on $L^2(E,\mu)$. So, we even have
	\begin{equation*}
		\int Lu \, v \, \mathrm{d}\mu = \int u \, Lv \, \mathrm{d}\mu \ \text{for all} \ u,v \in \mathcal{F} C_b^{\infty} (K) \text{,}
	\end{equation*}
	which is stronger than (\ref{5.19}).

\end{rem}

%%%%%%%%%%%%%%%%%%%%%%%%%%%%%%%%%%%%%%%%%%%%%%%%%%%%%%%%%%%%%%%%%%%%%%%%%%%%%%%%%%%%%%%%%%%%%%%%%%%%
%%%%%%%%%%%%%%%%%%%%%%%%%%%%%%%%%%%%%%%%%%%%%%%%%%%%%%%%%%%%%%%%%%%%%%%%%%%%%%%%%%%%%%%%%%%%%%%%%%%%
%%%%%%%%%%%%%%%%%%%%%%%%%%%%%%%%%%%%%%%%%%%%%%%%%%%%%%%%%%%%%%%%%%%%%%%%%%%%%%%%%%%%%%%%%%%%%%%%%%%%

\section{Acknowledgements}

It is a special pleasure for us to contribute to this volume in Honour of Professor Fukushima. He has been a great mentor for us, providing 
us much inspiration over many years. We take the opportunity to express to him our sincere gratitude, our friendship and great admiration. 
We also would like to thank Zdzislaw Brzezniak for pointing out an error in an earlier version of this paper. Financial support of the DFG 
through SFB 701, IGK 1132, and the HCM, as well as of the NCMIS, 973 project (2011CB808000), NSFC (11021161) and NSERC (Grant No. 311945-2013)
 is gratefully \\acknowledged. 

%%%%%%%%%%%%%%%%%%%%%%%%%%%%%%%%%%%%%%%%%%%%%%%%%%%%%%%%%%%%%%%%%%%%%%%%%%%%%%%%%%%%%%%%%%%%%%%%%%%%
%%%%%%%%%%%%%%%%%%%%%%%%%%%%%%%%%%%%%%%%%%%%%%%%%%%%%%%%%%%%%%%%%%%%%%%%%%%%%%%%%%%%%%%%%%%%%%%%%%%%
%%%%%%%%%%%%%%%%%%%%%%%%%%%%%%%%%%%%%%%%%%%%%%%%%%%%%%%%%%%%%%%%%%%%%%%%%%%%%%%%%%%%%%%%%%%%%%%%%%%%

%%%%%%%%%%%%%%%%%%%%%%%%%%%%%%%%%%%%%%%%%%%%%%%%%%%%%%%%%%%%%%%%%%%%%%%%%%%%%%%%%%%%%%%%%%%%%%%%%%%%
%%%%%%%%%%%%%%%%%%%%%%%%%%%%%%%%%%%%%%%%%%%%%%%%%%%%%%%%%%%%%%%%%%%%%%%%%%%%%%%%%%%%%%%%%%%%%%%%%%%%
%%%%%%%%%%%%%%%%%%%%%%%%%%%%%%%%%%%%%%%%%%%%%%%%%%%%%%%%%%%%%%%%%%%%%%%%%%%%%%%%%%%%%%%%%%%%%%%%%%%%


\begin{thebibliography}{GKLR01}

  \bibitem[Aid00]{Aid00}
\textsc{S. Aida}, 
 \textit{On the irreducibility of Dirichlet forms on domains in infinite-dimensional spaces},
 Osaka J. Math. {\bf 37} (2000), no.4, 953-966.
 
   \bibitem[Alb97]{Alb97}
\textsc{S. Albeverio}, 
 \textit{Wiener and Feynman-Path integrals and their Applications},
 AMS, Proceedings of Symposia in Applied Mathematics {\bf 52} (1997), 163-194.
 
    \bibitem[Alb03]{Alb03}
\textsc{S. Albeverio}, 
 \textit{Theory of Dirichlet forms and applications},
 In: P. Bernard (ed.) Lectures on probability theory and statistics (Saint-Flour, 2000), In: Lecture Notes in Math. {\bf 1816} (2003), 1-106, Springer, Berlin.
 
 \bibitem[ABR99]{ABR99}
\textsc{S. Albeverio, V. I. Bogachev and  M. R\"ockner}, 
 \textit{On uniqueness of invariant measures for finite and infinite dimensional diffusions},
 Comm. Pure Appl. Math.  {\bf 52} (1999), 325-362.
 
  \bibitem[AFHKL86]{AFHKL86}
\textsc{S. Albeverio, J.E. Fenstad, R. H{\o}egh-Krohn and T. Lindstr{\o}m}, 
 \textit{Nonstandard methods in stochastic analysis and mathematical physics},
 Pure and Applied Mathematics {\bf 122} (1986), Academic Press Inc., Orlando, FL. Reprint: Dover Publications, Mineola, NY (2009).
 
    \bibitem[AFe04]{AFe04}
\textsc{S. Albeverio and B. Ferrario}, 
 \textit{Uniqueness of solutions of the stochastic Navier-Stokes equation with invariant measure given by the enstrophy},
 Ann. Probab. {\bf 32}(2) (2004), 1632-1649. 
 
    \bibitem[AFe08]{AFe08}
\textsc{S. Albeverio and B. Ferrario}, 
 \textit{Some methods of infinite dimensional analysis in hydrodynamics: an introduction},
 SPDE in hydrodynamic: recent progress and prospects, Lecture Notes in Math. {\bf 1942} (2008), 1-50. 
 
   \bibitem[AFH11]{AFH11}
 \textsc{S. Albeverio, R. Fan and F.S. Herzberg},
 \textit{Hyperfinite Dirichlet Forms and Stochastic Processes},
 Lecture Notes of the Unione Matematica Italiana {\bf 10} (2011), 1-63.
 
    \bibitem[AFHMR92]{AFHMR92}
\textsc{S. Albeverio, M. Fukushima, W. Hansen, Z.M. Ma and M. R{\"o}ckner}, 
 \textit{An invariance result for capacities on Wiener space},
 J. Funct. Anal {\bf 106} (1992), 35-49.
   
   \bibitem[AGHK79]{AGHK79}
\textsc{S. Albeverio, G. Gallavotti and R. H{\o}egh-Krohn}, 
 \textit{Some results for the exponential interaction in two or more dimensions},
 Comm. Math. Phys. {\bf 70} (1979), 187-192.
 
    \bibitem[AGW97]{AGW97}
\textsc{S. Albeverio, H. Gottschalk and J.-L. Wu}, 
 \textit{Models of local relativistic quantum fields with indefinite metric (in all dimensions)},
 Comm. Math. Phys. {\bf 184}(3) (1997), 509-531.
 
    \bibitem[AGY05]{AGY05}
\textsc{S. Albeverio, H. Gottschalk and M.W. Yoshida}, 
 \textit{Systems of classical particles in the grand canonical ensemble, scaling limits and quantum field theory},
 Rev. Math. Phys. {\bf 17}(2) (2005), 175-226.
 
    \bibitem[AHR01]{AHR01}
\textsc{S. Albeverio, Z. Haba and F. Russo}, 
 \textit{A two-space dimensional semilinear heat equation perturbed by (Gaussian) white space},
 PTRF {\bf 121} (2001), 319-366.
 
     \bibitem[AHK74]{AHK74}
\textsc{S. Albeverio and R. H{\o}egh-Krohn}, 
 \textit{The Wightman axioms and the mass gap for strong interactions of exponential type in two-dimensional space-time},
 J. Funct. Anal. {\bf 16} (1974), 39-82.
 
   \bibitem[AHK76]{AHK76}
\textsc{S. Albeverio and R. H{\o}egh-Krohn}, 
 \textit{Quasi invariant measures, symmetric diffusion processes and quantum fields},
 Editions du CNRS, Proceedings of the International Colloquium on Mathematical Methods of Quantum Field Theory, Colloques Internationaux du Centre National de la Recherche Scientifique, No. {\bf 248} (1976), 11-59.
 
   \bibitem[AHK77a]{AHK77a}
\textsc{S. Albeverio and R. H{\o}egh-Krohn}, 
 \textit{Dirichlet forms and diffusion processes on rigged Hilbert spaces},
 Zeitschrift f{\"u}r Wahrscheinlichkeitstheorie und verwandte Gebiete {\bf 40} (1977), 1-57.
 
   \bibitem[AHK77b]{AHK77b}
\textsc{S. Albeverio and R. H{\o}egh-Krohn}, 
 \textit{Hunt processes and analytic potential theory on rigged Hilbert spaces},
 Ann. Inst. H. Poincar\'{e} (Probability Theory) {\bf 13} (1977), 269-291.
 
     \bibitem[AHK82]{AHK82}
\textsc{S. Albeverio and R. H{\o}egh-Krohn}, 
 \textit{Diffusions, quantum fields and groups of mappings},
 Functional analysis in Markov processes, Lecture Notes in Math. {\bf 923} (1982), 133-145. 
 
   \bibitem[AHKS77]{AHKS77}
\textsc{S. Albeverio, R. H{\o}egh-Krohn and L. Streit}, 
 \textit{Energy forms, Hamiltonians, and distorted Brownian paths},
 J. Math. Phys. {\bf 18} (1977), 907-917.
 
   \bibitem[AHPRS90a]{AHPRS90a}
\textsc{S. Albeverio, T. Hida, J. Potthoff, M. R{\"o}ckner and L. Streit}, 
 \textit{Dirichlet forms in terms of white noise analysis I - Construction and QFT examples},
 Rev. Math. Phys {\bf 1} (1990), 291-312.
 
   \bibitem[AHPRS90b]{AHPRS90b}
\textsc{S. Albeverio, T. Hida, J. Potthoff, M. R{\"o}ckner and L. Streit}, 
 \textit{Dirichlet forms in terms of white noise analysis II - Closability and Diffusion Processes},
 Rev. Math. Phys {\bf 1} (1990), 313-323.

   \bibitem[AKaR{\"o}12]{AKaR12}
\textsc{S. Albeverio, H. Kawabi, M. R\"ockner}, 
 \textit{Strong uniqueness for both Dirichlet operators and stochastic dynamics for Gibbs measures on a path space with exponential interactions},
 JFA {\bf 262} (2012), no. 2,  602-638. 
  
     \bibitem[AKKR09]{AKKR09}
\textsc{S. Albeverio, Y. Kondratiev, Y. Kozitsky and M. R{\"o}ckner}, 
 \textit{Statistical Mechanics of Quantum Lattice Systems: A Path Integral Approach},
 EMS Tracts Math. {\bf 8} (2009), European Mathematical Society.
 
      \bibitem[AKMR]{AKMR}
\textsc{S. Albeverio, H. Kawabi, S. Mihalache and M. R{\"o}ckner}, 
 \textit{Dirichlet form approach to stochastic quantization under exponential interaction in finite volume},
 in preparation.
 
  \bibitem[AKR92]{AKR92}
\textsc{S. Albeverio, Y. Kondratiev and  M. R\"ockner}, 
 \textit{An approximative criterium of essential selfadjointness of Dirichlet operators},
 Potential Anal. {\bf 1}(3) (1992), 307-317.
 
  \bibitem[AKR93]{AKR93}
\textsc{S. Albeverio, Y. Kondratiev and  M. R\"ockner}, 
 \textit{Addendum to the paper `An approximative criterium of essential selfadjointness of Dirichlet operators'},
 Potential Anal. {\bf 2} (1993), 195-198.
 
 \bibitem[AKR97a]{AKR97a}
\textsc{S. Albeverio, Y.G. Kondratiev and  M. R\"ockner}, 
 \textit{Ergodicity of $L^2$-semigroups and extremality of Gibbs states},
 J. Funct. Anal. {\bf 144} (1997), 3394-423.

 \bibitem[AKR97b]{AKR97b}
\textsc{S. Albeverio, Y.G. Kondratiev and  M. R\"ockner}, 
 \textit{Ergodicity for the stochastic dynamics of quasi-invariant measures with applications to Gibbs states},
 J. Funct. Anal. {\bf 149} (1997), 415-469.
 
  \bibitem[AKR98]{AKR98}
\textsc{S. Albeverio, Y.G. Kondratiev and  M. R\"ockner}, 
 \textit{Analysis and geometry on configuration space. The Gibbsian case},
 J. Funct. Anal. {\bf 157} (1998), 242-291.
 
   \bibitem[AKR02]{AKR02}
\textsc{S. Albeverio, Y.G. Kondratiev and  M. R\"ockner}, 
 \textit{Symmetryzing Measures for Infinite Dimensional Diffusions: An Analytic Approach},
 in Lect. Notes Math., ``Geometric Analysis and nonlinear partial differential equations'', S. Hildebrandt et. al., eds, Springer, Berlin, (2003), 475-486.
 
     \bibitem[AL08]{AL}
\textsc{S. Albeverio, S. Liang}, 
 \textit{A remark on the nonequivalence of the time-zero $\phi_3^4$-mesure with the free field measure},
 Markov Processes Rel. Fields {\bf 14} (2008), 159-164.
 
     \bibitem[ALZ06]{ALZ06}
\textsc{S. Albeverio, S. Liang and B. Zegarli\'{n}ski}, 
 \textit{Remark on the integration by parts formula for the $\phi_3^4$-quantum field model},
 Infin. Dimens. Anal. Quantum Probab. Relat. Top. {\bf 9} (2006), 149-154.
 
      \bibitem[AM91]{AM91}
\textsc{S. Albeverio, Z.M. Ma}, 
 \textit{Necessary and sufficient conditions for the existence of m- processes associeted with Dirichlet form},
S\'em. de Prob. XXV, LNM 148 5, 374-406 (1991)
 
     \bibitem[AM92]{AM92}
\textsc{S. Albeverio, Z.M. Ma}, 
 \textit{A general correspondence between Dirichlet form and right processes},
Bull AMS {\bf26} (1992), 245-252.
 
   \bibitem[AMR92a]{AMR92a}
\textsc{S. Albeverio, Z.M. Ma and  M. R\"ockner}, 
 \textit{Non-symmetric Dirichlet forms and Markov processes on general state space},
 CRA Sci., Paris {\bf 314} (1992), 77-82.
 
   \bibitem[AMR92b]{AMR92b}
\textsc{S. Albeverio, Z.M. Ma and  M. R\"ockner}, 
 \textit{Regularization of Dirichlet spaces and applications},
 CRAS, Ser. I, Paris {\bf 314}(11) (1992), 859-864.
 
   \bibitem[AMR92c]{AMR92c}
\textsc{S. Albeverio, Z.M. Ma and  M. R\"ockner}, 
 \textit{A Beurling-Deny structure theorem for Dirichlet forms on general state space},
 Cambridge University Press, R. H{\o}egh-Krohn's Memorial Volume, Eds. S. Albeverio et al., Vol {\bf 1} (1992), 115-123.
 
   \bibitem[AMR93]{AMR93}
\textsc{S. Albeverio, Z.M. Ma and  M. R\"ockner}, 
 \textit{Quasi-regular Dirichlet forms and Markov processes},
 J. Funct. Anal. {\bf 111} (1993), 118-154.
 
   \bibitem[AMU98]{AMU98}
\textsc{S. Albeverio, L.M. Morato and S. Ugolini}, 
 \textit{Non-Symmetric Diffusions and Related Hamiltonians},
 Potential Analysis {\bf 8} (1998), 195-204.
 
  \bibitem[AR89a]{AR89a}
\textsc{S. Albeverio and  M. R\"ockner}, 
 \textit{Dirichlet forms, quantum fields and stochastic quantization},
 Stochastic analysis, path integration and dynamics (K.D. Elworthy and J.C. Zambrini, eds.)
 Research Notes in Mathematics, vol. 200, Longman, Harlow, 1989, pp 1-21.
 
   \bibitem[AR89b]{AR89b}
\textsc{S. Albeverio and  M. R\"ockner}, 
 \textit{Classical Dirichlet forms on topological vector spaces - the construction of the associated diffusion process},
 Prob. Theory and Rel. Fields {\bf 83}(3) (1989), 405-434.
 
 \bibitem[AR90a]{AR90}
\textsc{S. Albeverio and  M. R\"ockner}, 
 \textit{Classical Dirichlet forms on topological vector spaces-Closability and Cameron-Martin formula},
 J. Funct. Anal. {\bf 88}(2) (1990), 395-436.
 
 \bibitem[AR90b]{AR90b}
 \textsc{S. Albeverio and  M. R\"ockner},
 \textit{New developments in theory and applications od Dirichlet forms}
 ''Stochastic Processes, Physics and Geometry, Ascona/Locarno, Switzerland, July 4-9, 1988'' pp.27-76,
 World Scientific, Singapore, 1990
 
  \bibitem[AR91]{AR91}
\textsc{S. Albeverio and  M. R\"ockner}, 
 \textit{Stochastic differential equations in infinite dimensions: solution via Dirichlet forms},
 Prob. Th. Rel. Fields {\bf 89}(3) (1991), 347-386.
 
    \bibitem[ARZ93a]{ARZ93a}
\textsc{S. Albeverio, M. R\"ockner and T.S. Zhang}, 
 \textit{Girsanov transform of symmetric diffusion processes in infinite dimensional space},
 Ann. of Prob. {\bf 21} (1993), 961-978.
 
   \bibitem[ARZ93b]{ARZ93b}
\textsc{S. Albeverio, M. R\"ockner and T.S. Zhang}, 
 \textit{Markov uniqueness and its applications to martingale problems, stochastic differential equation and stochastic quantization},
 C.R. Math. Rep. Acad. Sci. Canada XV, no. 1, 1-6.
 
 \bibitem[ARW01]{ARW}
\textsc{S. Albeverio, B. R\"udiger and J.-L. Wu}, 
\textit{Analytic and probabilistic aspects of L\'{e}vy processes and fields in quantum theory},
L\'{e}vy processes (2001), 187-224, Birkh\"auser Boston, Boston.
 
	\bibitem[AR{\"u}02]{AR02}
	\textsc{S. Albeverio and B. R\"udiger}, 
	\textit{Infinite dimensional Stochastic Differential Equations obtained by subordination and related Dirichlet forms},
  J. Funct. Anal. {\bf 204}(1) (2003), 122-156.
  
  	\bibitem[AR{\"u}05]{AR05}
	\textsc{S. Albeverio and B. R\"udiger}, 
	\textit{Subordination of symmetric quasi-regular Dirichlet forms},
  	Random Oper. Stochastic Equations {\bf 13}(1) (2005), 17-38. 
  	
  	 \bibitem[ASe14]{ASe14}
\textsc{S. Albeverio and A. Sengupta}, 
 \textit{From classical to quantum fields - A mathematical approach},
 book in preparation.

   \bibitem[AU00]{AU00}
  \textsc{S. Albeverio and S. Ugolini}, 
  \textit{Complex Dirichlet Forms: Non Symmetric Diffusion Processes and Schr\"odinger Operators},
  Potential Analysis {\bf 12} (2000), 403-417.
  
    \bibitem[BBCK09]{BBCK09}
\textsc{M.T. Barlow, R.F. Bass, Z.Q. Chen and Kassmann}, 
 \textit{Non-local Dirichlet forms and symmetric jump processes},
 TAMS {\bf 361} (2009), 1963-1999.
 
   \bibitem[BK95]{BK95}
  \textsc{Y.M. Berezansky and Y.G. Kondratiev}, 
  \textit{Spectral Methods in Infinite-Dimensional Analysis},
  Kluwer, Vol {\bf 1,2} (1995), first issued in URSS 1988.
 
      \bibitem[BD58]{BD58}
\textsc{A. Beurling and J. Deny}, 
 \textit{Espaces de Dirichlet},
 Acta Math. {\bf 99} (1958), 203-224.
 
      \bibitem[BD59]{BD59}
\textsc{A. Beurling and J. Deny}, 
 \textit{Dirichlet spaces},
 Proc. Nat. Acad. Sci. U.S.A. {\bf 45} (1959), 208-215.
 
  \bibitem[BeBoR06]{BeBoR06}
\textsc{L. Beznea, N. Boboc and  M. R\"ockner}, 
 \textit{Markov processes associated with $L^p$-resolvents and applications to stochastic differential equations on Hilbert space},
 J. Evol. Equ.  {\bf 6} (2006), no.4, 754-772.
 
   \bibitem[BeBoR08]{BeBoR08}
\textsc{L. Beznea, N. Boboc and  M. R\"ockner}, 
 \textit{Markov processes associated with $L^p$-resolvents,applications to quasi-regular Dirichlet forms and stochastic differential equations},
 C. R. Math. Acad. Sci. Paris {\bf 346}(5-6) (2008), 323-328.

     \bibitem[BCR14]{BCR14}
\textsc{L. Beznea, J. Cimpian and  M. R\"ockner}, 
 \textit{Irreducible recurrence, ergodicity and extremality of invariant measures for $L^p$-resolvents},
 pp.25, Preprint 2014.
 
  \bibitem[BhK93]{BhK93}
\textsc{A.G. Bhatt and R.L. Karandikar}, 
 \textit{Invariant measures and evolution equations for Markov processes characterized via martingale problems},
 Ann. Prob. {\bf 21} (1993), 2246-2268.
 
   \bibitem[BiT07]{BiT07}
\textsc{M. Biroli and N. Tchou}, 
 \textit{$\Gamma$-convergence for strongly local Dirichlet forms in perforated domain with homogeneous Neumann boundary conditions},
 Adv. Math. Sci. Appl. {\bf17} (2007), 149-179. 
 
   \bibitem[BKRS14]{BKRS14}
\textsc{V.I. Bogachev, N.V. Krylov,  M. R\"ockner and S. Shaposhnikov}, 
 \textit{Fokker-Planck Kolmogorov equations},
 Monograph, pp. 490, publication in preparation.
 
   \bibitem[BoR95]{BoR95}
\textsc{V.I. Bogachev and  M. R\"ockner}, 
 \textit{Mehler formula and capacities for infinite dimensional OrnsteinUhlenbeck processes with general linear drift},
 Osaka J. Math. {\bf 32} (1995), 237-274.
 
   \bibitem[BoR\"oZh00]{BoRZh}
\textsc{V. Bogachev, M. R{\"o}ckner and T.S. Zhang}, 
 \textit{Existence and uniqueness of invariant measures: an approach via sectorial forms},
 Appl. Math. Optim. {\bf 41}(1) (2000), 87-109.
  
    \bibitem[BRSt00]{BStR00}
\textsc{V.I. Bogachev,  M. R\"ockner and W. Stannat}, 
 \textit{Uniqueness of invariant measures and maximal dissipativity of diffusion operators on $L^1$},
 Infinite dimensional Stochastic Analysis (P. Cl\'ement et al., eds.), Royal Netherlands Academy of Arts and Sciences, Amsterdam, 2000, pp. 39-54.
 
       \bibitem[BCM88]{BCM88}
\textsc{V.S. Borkar, R.T. Chari and S.K. Mitter}, 
 \textit{Stochastic quantization of field theory in finite and infinite volume},
 J. Funct. Anal. {\bf 81}(1) (1988), 184-206.
  
   \bibitem[Bou03]{Bou03}
\textsc{N. Bouleau}, 
 \textit{Error calculus for finance and physics: the language of Dirichlet form, de Gruyter},
Berlin (2003). 
 
    \bibitem[BH91]{BH91}
\textsc{N. Bouleau and F. Hirsch}, 
 \textit{Dirichlet Forms and Analysis on Wiener Space},
 deGruyter, Studies in Mathematics {\bf 14} (1991).
 
   \bibitem[CF12]{CF12}
\textsc{Z.Q. Chen and M. Fukushima}, 
 \textit{Symmetric Markov Processes, Time change, and Boundary Theory},
 Princeton University Press, 2012.
   
  \bibitem[CMR94]{CMR94}
\textsc{Z.Q. Chen, Z.M. Ma and M.R{\"o}ckner}, 
 \textit{Quasi-homeomorphisms of Dirichlet forms},
 Nagoja J.Math. {\bf 136} (1994), 1-15.
 
     \bibitem[DPD03]{DPD}
\textsc{G. Da Prato and A. Debussche}, 
 \textit{Strong solutions to the stochastic quantization equations},
 Ann. Probab. {\bf 31}(4) (2003), 1900-1916.
 
      \bibitem[DPZ92]{DPZa}
\textsc{G. Da Prato and J. Zabczyk}, 
 \textit{Stochastic equations in infinite dimensions},
 Encyclopedia of Mathematics and its Applications {\bf 44}, Cambridge University Press, Cambridge (1992).
 
     \bibitem[DPT00]{DPT00}
\textsc{G. Da Prato and L. Tubaro}, 
 \textit{Self-adjointness of some infinite dimensional elliptic operators and applications to stochastic quantization},
 Prob. Theory Related Fields {\bf 118}(1) (2000), 131-145.
 
   \bibitem[Den70]{Den70}
\textsc{J. Deny}, 
 \textit{M\'{e}thodes hilbertiennes en th\'{e}orie du potentiel},
 In: Potential Theory (C.I.M.E., I Ciclo, Stresa, 1969) (1970), 121-201, Edizioni Cremonese, Rome.
 
    \bibitem[DG97]{DG97}
\textsc{Z. Dong and F.Z. Gong}, 
 \textit{A necessary and sufficient condition for quasiregularity of Dirichlet forms corresponding to resolvent kernels},
 Acta Math. Appl. Sinica {\bf 20}(3) (1997), 378-385.

    \bibitem[E99]{E99}
\textsc{A. Eberle}, 
 \textit{Uniqueness and Non-Uniqueness of Semigroups generated by Singular Diffusion Operators},
 Lecture Notes in Math. Vol. 1718, Springer, Berlin 1999.

     \bibitem[Eche82]{Eche}
\textsc{P.E. Echeverria}, 
 \textit{A criterium for invariant measures of Markov processes},
 W. Th. Vens. Geb. {\bf 61} (1982), 1-16.
 
   \bibitem[EK86]{EK86}
\textsc{S.N. Ethier and T.G. Kurtz}, 
 \textit{Markov Processes. Characterization and Convergence},
 Wiley Series in Probability and Mathematical Statistics (1986).
 
    \bibitem[FeyLa91]{FeyLa91}
\textsc{D. Feyel and A. de la Pradelle}, 
 \textit{Capacit\'{e}s gaussiennes},
 Ann. Inst. Fourier {\bf 41} (1991), 49-76.
 
    \bibitem[Fr\"o74]{Fro}
\textsc{J. Fr\"ohlich}, 
 \textit{Verification of axioms for Euklidean and relativistic fields and Haag's theorem in a class of $P(\varphi)_2$ models},
 Ann. I. H. Poincar\'{e} {\bf 21} (1974), 271-317.
 
     \bibitem[FrP77]{FrP77}
\textsc{J. Fr\"ohlich  and T.M. Park}, 
 \textit{Remarks on exponential interactions and the quantum sine-Gordon equation in two space-time dimensions},
 HPA {\bf 50} (1977), 315-329.
  
   \bibitem[Fu71a]{Fu71a}
\textsc{M. Fukushima}, 
 \textit{Regular representations of Dirichlet spaces},
 TAMS {\bf 155} (1971), 455-473.
 
   \bibitem[Fu71b]{Fu71b}
\textsc{M. Fukushima}, 
 \textit{Dirichlet spaces and strong Markov processes},
 TAMS {\bf 162} (1971), 185-224. 
 
    \bibitem[Fu80]{Fu80}
\textsc{M. Fukushima}, 
 \textit{Dirichlet Forms and Markov processes},
 North Holland (1980), Amsterdam.
 
    \bibitem[F82]{F82}
 \textsc{M. Fukushima},
 \textit{A note on irreducibility and ergodicity of symmetric Markov processes},
 in ''Proceedings Marseille'', Lecture Notes in Physics, Vol. 173, pp. 200-207, Springer, Berlin, 1982.
 
    \bibitem[Fu84]{Fu84}
\textsc{M. Fukushima}, 
 \textit{A Dirichlet form on the Wiener space and properties of Brownian motion},
 Lecture Notes in Math. Vol. 1096, 290-300, Springer, 1984.
 
    \bibitem[Fu92]{Fu92}
\textsc{M. Fukushima}, 
 \textit{$(r,p)$ capacities and Hunt processes in infinite dimensions},
 in A.N. Shiryaev et al., eds, 96-103, Probability, Theory and Mathematical Statistics (1992), World Scient, Singapore.
 
   \bibitem[Fu10]{Fu10}
\textsc{M. Fukushima}, 
 \textit{From one dimensional diffusions to symmetric Markov processes},
Stoch. Proc. Appl. {\bf 120} (2010), 590-604.
 
    \bibitem[FH01]{FH01}
\textsc{M. Fukushima and M. Hino}, 
 \textit{On the space of BV functions and a related stochastic calculus in infinite dimensions},
 J. Funct. Anal. {\bf 183} (2001), 245-268.
 
    \bibitem[FOT11]{FOT11}
\textsc{M. Fukushima, Y. Oshima and M. Takeda}, 
 \textit{Dirichlet forms and symmetric Markov processes},
 2nd revised and extended ed., de Gruyter Studies in Mathematics {\bf 19} (2011), Walter de Gruyter \& Co., Berlin.
 
   \bibitem[GaGo96]{GaGo}
 \textsc{D. Gatarek and B. Goldys},
 \textit{Existence, uniqueness and ergodicity for the stochastic quantization equation},
Studia Math. {\bf 119} (1996), 179-193. 

    \bibitem[GJ81]{GJ81}
 \textsc{J. Glimm and A. Jaffe},
 \textit{Quantum Physics},
 Springer, New York, 1981; 2nd ed., 1986.
 
   \bibitem[GriH08]{GriH08}
 \textsc{A. Grigor\'yan and J. Hu},
 \textit{Off-diagonal upper estimates for the heat normal of the Dirichlet forms on metric spaces},
Inv. Math. {\bf 174} (2008), 81-126.
 
   \bibitem[Gro76]{Gro76}
\textsc{L. Gross}, 
 \textit{Logarithmic Sobolev inequalities},
 Amer. J. Math. {\bf 97} (1976), 1061-1083.
 
  \bibitem[GRS75]{GRS75}
\textsc{F. Guerra, J. Rosen and B. Simon}, 
 \textit{The $P(\Phi)_2$-Euclidean quantum field theory as classical statistical mechanics},
 Ann. Math. {\bf 101} (1975), 111-259.
 
    \bibitem[Hai14]{Hai14}
\textsc{M. Hairer}, 
 \textit{A theory of regularity structures},
 ArXive-prints (2013), arXiv: 1303.51113.
 
   \bibitem[HKPS93]{HKPS93}
\textsc{T. Hida, H.-H. Kuo, J. Potthoff and L. Streit}, 
 \textit{White Noise: An Infinite Dimensional Calculus},
 Math. Appl. {\bf 253} (1993), Kluwer Academic Publishers Group, Dordrecht.
 
   \bibitem[Hi10]{Hi10}
\textsc{M. Hino}, 
 \textit{Energy measures and indices of Dirichlet forms, with applications to derivatives on some fractals},
Proc. Lond. Math. foc. {\bf 100} (2010), 269-302.
 
    \bibitem[HK71]{HK71}
\textsc{R. H{\o}egh-Krohn}, 
 \textit{A general class of quantum fields without cut-offs in two space-time dimensions},
 Comm. Math. Phys. {\bf 21} (1971), 244-255.
 
   \bibitem[HK98]{HK98}
 \textsc{Y. Hu and G. Kallianpur},
 \textit{Exponential Integrability and Application to Stochastic Quantization},
 Appl. Math. Optim. {\bf 37} (1998), 295-353.
 
   \bibitem[Iwa85]{Iwa85}
\textsc{K. Iwata}, 
 \textit{Reversible measures of a $P(\phi)_1$-time evolution},
 Probabilistic methods in mathematical physics (1985), 195-209.
 
    \bibitem[Iwa87]{Iwa87}
\textsc{K. Iwata}, 
 \textit{An infinite dimensional stochastic differential equation with state space $C(\mathbb{R})$},
 Probab. Theory Related Fields {\bf 74} (1987), 141-159.
 
   \bibitem[Jac01]{Jac01}
\textsc{N. Jacob}, 
 \textit{Pseudo-Differential Operators aand Markov Processes},
 Vol. I. Fourier analysis and semigroups (2001), Imperial College Press, London.
 
    \bibitem[JS00]{JS00}
\textsc{N. Jacob and R.L. Schilling}, 
 \textit{Fractional derivatives, non-symmetric and time-dependent Dirichlet forms and the drift form},
 Z. Anal. Anwendungen {\bf 19}(3) (2000), 801-830.
 
    \bibitem[JLM85]{JLM85}
\textsc{G. Jona-Lasinio and P.K. Mitter}, 
 \textit{On the stochastic quantization of field theory},
 Comm. Math. Phys. {\bf 101} (1985), 409-436.
 
  \bibitem[Jo95]{Jo}
\textsc{R. Jost}, 
 \textit{Foundations of Quantum Field Theory},
 pp. 153-169 in R.\ Jost, Das M{\"a}rchen vom Elfenbeinernen Turm, Hrsg.\ K.\ Hepp et al., Springer, Berlin (1995).

      \bibitem[KaYa05]{KaYa05}
\textsc{H. Kaneko, K. Yasuda}, 
 \textit{Capacities associated with Dirichlet space on an infinite extension of a local field},
 Forum Math. {\bf 17} (2005), 1011-1032.
 
     \bibitem[KKVW09]{KKVW09}
\textsc{V. Kant, T. Klauss, J. Voigt, M. Weber}, 
 \textit{Dirichlet forms for singular one-dimensional operators and on graphs},
 J. Evol. Equ. {\bf 9} (2009), 637-659.

   \bibitem[KaStr14]{KaStr14}
\textsc{W. Karwowski and L. Streit}, 
 \textit{contribution to this volume},
 .
 
    \bibitem[KR07]{KR07}
\textsc{H. Kawabi and M. R{\"o}ckner}, 
 \textit{Essential self-adjointness of Dirichlet operators on a path space with Gibbs measures via an SPDE approach},
 J. Funct. Anal {\bf 242} (2007), 486-518.

  \bibitem[Kol06]{Kol06}
\textsc{A.V. Kolesnikov}, 
 \textit{Mosco convergence of Dirichlet forms in infinite dimensions with changing reference measures},
 JFA {\bf 230} (2006), 382-418. 
 
    \bibitem[KT91]{KT91}
\textsc{Y. Kondratiev and T.V. Tsycalenko}, 
 \textit{Dirichlet operators and associated differential equations},
 Selecta Math. Sovietica {\bf 10}(4) (1991), 345-397.
 
   \bibitem[Ku92]{Ku}
\textsc{S. Kusuoka}, 
 \textit{H{\o}egh-Krohn's model of quantum fields and the absolute continuity of measures},
 pp. 405-424 in S. Albeverio, J.E. Fenstad, H. Holden, T. Lindstr{\o}m, eds., Ideas and Methods in Mathematical Analysis, stochastics, and Applications, Cambridge Univ. Press (1992).
 
    \bibitem[Kus82]{Kus82}
\textsc{S. Kusuoka}, 
 \textit{Dirichlet forms and diffusion processes on Banach spaces},
 J. Fac. Sc. Univ. Tokyo {\bf 29} (1982), 79-95.
 
     \bibitem[KS03]{KS03}
\textsc{K. Kuwae and T. Shioya}, 
 \textit{Sobolev and Dirichlet spaces over maps between metric spaces},
 J. Reine Angew. Math. {\bf 555} (2003), 39-75.
 
   \bibitem[LW99]{LW99}
\textsc{A. Laptev and T. Weidl}, 
 \textit{Hardy inequalities for magnetic Dirichlet forms},
 Oper. Theory Adv. Appl. {\bf 108} (1999), 299-305.
 
 \bibitem[LM72]{LM72}
 \textsc{J.L. Lions and E. Magenes},
 \textit{Non-homogeneous boundary value problems and applications},
Grundlehren Math. Wiss., Springer, Berlin, 1972.

  \bibitem[LR98]{LR98}
\textsc{V. Liskevich and M. R{\"o}ckner}, 
 \textit{Strong Uniqueness for Certain Infinite Dimensional Dirichlet Operators and Applications to Stochastic Quantization},
 Ann. Scuola Norm. Pisa {\bf 27}(1) (1998), 69-91.

  \bibitem[Ma95]{Ma95}
 \textsc{Z.M. Ma},
 \textit{Quasiregular Dirichlet forms and applications},
 Proc. ICM 1994, Z\"urich, Switzerland, 1994, Birkh\"auser, Basel (1995).
  
   \bibitem[MR92]{MR92}
 \textsc{Z.M. Ma and M. R\"ockner},
 \textit{Introduction to the theory of (non-symmetric) Dirichlet forms},
 Springer, Berlin, 1992
 
   \bibitem[MR10]{MR10}
 \textsc{C. Marinelli and M. R\"ockner},
 \textit{On uniqueness of mild solutions for dissipative stochastic evolution equations},
 IDAQP {\bf 13} (2010), no. 3, 363-376.
 
     \bibitem[MZ10]{MZ10}
\textsc{C. Marinelli and G. Ziglio}, 
 \textit{Ergodicity for nonlinear stochastic evolution equations with multiplicative Poisson noise},
 Dyn. PDE {\bf 7} (2010), 1-23. 
  
     \bibitem[Mi06]{Mi}
\textsc{S. Mihalache}, 
 \textit{Stochastische Quantisierung bei exponentieller Wechselwirkung},
Diploma Thesis, Bonn (2006).
 
    \bibitem[MiRo99]{MiRo}
\textsc{R. Mikulevicius and B. Rozovskii}, 
 \textit{Martingale problems for stochastic PDE's},
 pp. 243-325 in Stochastic Partial Differential Equations: Six Perspectives, R. Carmona and B. Rozovskii, eds., AMS (1999).
  
   \bibitem[Mi89]{Mit89}
\textsc{S. Mitter}, 
 \textit{Markov random fields, stochastic quantization and image analysis},
 Math. Appl. {\bf 56} (1989), 101-109.
 
 \bibitem[N73]{N73}
 \textsc{E. Nelson},
 \textit{The free Markov field}
 J. Funct. Anal. {\bf 12} (1973), 211-227

  \bibitem[Osh04]{Osh04}
 \textsc{Y. Oshima},
 \textit{Time-dependent Dirichlet forms and related stochastic calculus},
 IDAQP {\bf 7} (2004), 281-316.
 
   \bibitem[PaWu81]{PaWu81}
\textsc{G. Parisi and Y.S. Wu}, 
 \textit{Perturbation theory without gauge fixing},
 Sci. Sinica {\bf 24}  (1981), 483-496. 
 
   \bibitem[PeZ07]{PeZ}
 \textsc{S. Peszat and J. Zabczyk},
 \textit{Stochastic partial differential equations with L\'{e}vy noise. An evolution equation approach},
 Encyclopedia of Mathematics and its Applications {\bf 113}, Cambridge University Press, Cambridge (2007).

  \bibitem[RS75]{RS78}
 \textsc{M. Reed and B. Simon},
 \textit{Methods of Modern Mathematical Physics. II: Fourier Analysis, Self-Adjointness},
 Academic Press, Orlando, FL, 1975
 
  \bibitem[RhVa13]{RhVa13}
 \textsc{R. Rhodes, V. Vargas},
 \textit{Gaussian multiplicative chaos and applications},
a seriens arXiv (2013)
 
 \bibitem[R86]{R86}
 \textsc{M. R\"ockner},
 \textit{Specifications and Martin boundaries for $P(\Phi)_2$-random fields},
 Commun. Math. Phys. {\bf 106} (1986), 105-135.
 
  \bibitem[R88]{R88}
 \textsc{M. R\"ockner},
 \textit{Traces of harmonic functions and a new path space for the free quantum field},
 J. Funct. Anal. {\bf 79} (1988), 211-249.
 
 \bibitem[R98]{R98}
 \textsc{M. R\"ockner},
 \textit{$L^p$-analysis of finite and infinite dimensional diffusion operators},
 Stochastic PDE's and Kolmogorov's equations in infinite dimensions (Giuseppe Da Prato, ed.), Lect. Notes Math., vol. 1715, Springer, Berlin, 1999, pp. 401-408.
 
  \bibitem[RS92]{RS92}
 \textsc{M. R\"ockner and B. Schmuland},
 \textit{Tightness of general $C_{1,p}$-capacities on Banach space},
 J. Funct. Anal. {\bf 108} (1992), 1-12
 
   \bibitem[RZ92]{RZ92}
 \textsc{M. R\"ockner and T.S. Zhang},
 \textit{Uniqueness of generalized Schr\"odinger operators and applications},
 J. Funct. Anal. {\bf 105} (1992), 187-231
 
    \bibitem[SU07]{SU07}
\textsc{R.L. Schilling and T. Uemura}, 
 \textit{On the Feller property of Dirichlet forms generated by pseudo differential operators},
 Tohoku Math. J. (2) {\bf 59}(3) (2007), 401-422.
 
   \bibitem[SS03]{SS03}
\textsc{B. Schmuland and W. Sun}, 
 \textit{The maximum Markovian self-adjoint extensions of Dirichlet operators for interacting particle systems},
 Forum Math. {\bf 15}(4) (2003), 615-638.
 
   \bibitem[ST92]{ST92}
\textsc{I. Shigekawa and S. Taniguchi}, 
 \textit{Dirichlet forms on separable metric spaces},
 Ed. A.N. Shiryaev et al., World Scient. Singapore, Probability Theory and Mathematical Statistics (1992), 324-353.
 
  \bibitem[Sil74]{Sil74}
\textsc{M.L. Silverstein}, 
 \textit{Symmetric Markov Processes},
 Lect. Notes in Maths. {\bf 426} (1974), Springer, Berlin.
 
  \bibitem[Sil76]{Sil76}
\textsc{M.L. Silverstein}, 
 \textit{Boundary Theory for Symmetric Markov Processes},
 Lect. Notes in Maths. {\bf 516} (1976), Springer, Berlin.
 
   \bibitem[Si74]{S74}
 \textsc{B. Simon},
 \textit{The $P(\Phi)_2$-Euclidean (Quantum) Field Theory},
 Princeton Univ. Press, Princeton, NJ, 1974
 
  \bibitem[Sta99a]{Sta99}
\textsc{W. Stannat}, 
 \textit{Theory of generalized Dirichlet forms and its applications in analysis and stochastics},
 Am. Math. Soc. {\bf 142}(678)  (1999), 1-101.
 
   \bibitem[Sta99b]{Sta99b}
\textsc{W. Stannat}, 
 \textit{(Nonsymmetric) Dirichlet operators on $L^1$: existence, uniqueness and associated Markov processes},
 Annali. Scuola Norm. Sup. di Pisa Cl. Sci (4) {\bf 28} (1999), 99-140.
 
 \bibitem[StV06]{StV}
\textsc{D.W. Stroock and S.R.S. Varadhan}, 
\textit{Multidimensional diffusion processes},
Reprint of the 1997 edition, Classics in Mathematics, Springer, Berlin (2006).
 
   \bibitem[Tak92]{Tak92}
\textsc{M. Takeda}, 
 \textit{The maximum Markovian self-adjoint extensions of generalized Schr{\"o}dinger operators},
 J. Math. Soc. Jap. {\bf 44} (1992), 113-130.
 
 \bibitem[Y89]{Y89}
 \textsc{J. A. Yan}
 \textit{Generalizations of Gross' and Minlos' theorems},
 In: Azema, J., Meyer, P.A., Yor, M. (eds.) S\'eminaire de Probabilit\'es. XXII(Lect. Notes Math., vol. 1372, pp. 395-404)
 Berlin Heidelberg New York: Springer 1989
 
\end{thebibliography}
\end{document}